\documentclass{amsart}
\usepackage{mathtools}
\usepackage{amsmath,amssymb,amsthm,color,mathrsfs,array}
\usepackage{multirow}
\usepackage{verbatim}
\usepackage{url}

\newcommand{\rev}[1]{{\color{black}#1}}

\newcommand{\re}{\mathbb{R}}
\newcommand{\mR}{\mathbb{R}}

\newcommand{\mC}{\mathbb{C}}

\newcommand{\cpx}{\mathbb{C}}

\newcommand{\N}{\mathbb{N}}

\newcommand{\diag}{\mbox{diag}}

\newcommand{\nn}{\nonumber}

\def\af{\alpha}

\def\rank{\mbox{rank}}

\newcommand{\sig}{\sigma}
\newcommand{\Sig}{\Sigma}

\newcommand{\reff}[1]{(\ref{#1})}

\newcommand{\mc}[1]{\mathcal{#1}}

\newcommand{\supp}[1]{\mbox{supp}(#1)}

\newcommand{\qmod}[1]{\mbox{QM}[#1]}
\newcommand{\ideal}[1]{\mbox{Ideal}[#1]}
\newcommand{\st}{\mathit{s.t.}}

\newcommand{\bdes}{\begin{description}}
	\newcommand{\edes}{\end{description}}
\newcommand{\bal}{\begin{align}}
	\newcommand{\eal}{\end{align}}
\newcommand{\bnum}{\begin{enumerate}}
	\newcommand{\enum}{\end{enumerate}}
\newcommand{\bit}{\begin{itemize}}
	\newcommand{\eit}{\end{itemize}}
\newcommand{\bea}{\begin{eqnarray}}
	\newcommand{\eea}{\end{eqnarray}}
\newcommand{\be}{\begin{equation}}
	\newcommand{\ee}{\end{equation}}
\newcommand{\baray}{\begin{array}}
	\newcommand{\earay}{\end{array}}
\newcommand{\bsry}{\begin{subarray}}
	\newcommand{\esry}{\end{subarray}}
\newcommand{\bca}{\begin{cases}}
	\newcommand{\eca}{\end{cases}}
\newcommand{\bcen}{\begin{center}}
	\newcommand{\ecen}{\end{center}}
\newcommand{\bbm}{\begin{bmatrix}}
	\newcommand{\ebm}{\end{bmatrix}}
\newcommand{\bmx}{\begin{matrix}}
	\newcommand{\emx}{\end{matrix}}
\newcommand{\bpm}{\begin{pmatrix}}
	\newcommand{\epm}{\end{pmatrix}}
\newcommand{\btab}{\begin{tabular}}
	\newcommand{\etab}{\end{tabular}}
\newcommand{\ter}[1]{\textcolor{red}{#1}}

\theoremstyle{plain}
\newtheorem{theorem}{Theorem}[section]

\newtheorem{prop}[theorem]{Proposition}
\newtheorem{lem}[theorem]{Lemma}

\newtheorem*{claim*}{Claim}
\newtheorem{thm}[theorem]{Theorem}
\theoremstyle{definition}

\newtheorem{exm}[theorem]{Example}

\numberwithin{equation}{section}
\numberwithin{table}{section}

\begin{document}
%\title[Finite convergence of the Moment-SOS relaxations...]
%{Finite convergence of the Moment-SOS relaxations for generalized moment problem}

\title[Finite convergence of the matrix Moment-SOS hierarchy]
{Finite convergence of the Moment-SOS hierarchy for polynomial matrix optimization}

\author[Lei Huang]{Lei Huang}
\address{Lei Huang,Department of Mathematics,
	University of California San Diego,
	9500 Gilman Drive, La Jolla, CA, USA, 92093.}
\email{huanglei@lsec.cc.ac.cn}

\author[Jiawang Nie]{Jiawang~Nie}
\address{Jiawang Nie,  Department of Mathematics,
University of California San Diego,
9500 Gilman Drive, La Jolla, CA, USA, 92093.}
\email{njw@math.ucsd.edu}

%\author[Ya-Xiang Yuan]{Ya-Xiang Yuan}
%\address{Ya-Xiang Yuan,
 %Institute of Computational Mathematics and Scientific/Engineering Computing, Academy of Mathematics and Systems Science, Chinese Academy of Sciences, Beijing, China, 100049.}
%\email{yyx@lsec.cc.ac.cn}

\subjclass[2020]{65K05,, 90C22,90C23 90C26}

\keywords{polynomial optimization, homogenization,
Moment-SOS relaxations, optimality conditions}

\maketitle

%\bit
%\item Corresponding Author: Lei Huang
%\item Affiliation:   Institute of Computational Mathematics and Scientific/Engineering Computing, Academy of Mathematics and Systems Science, Chinese Academy of Sciences,  and School of Mathematical Sciences, University of Chinese Academy of Sciences, Beijing, China, 100049
%\item E-mail Address: huanglei@lsec.cc.ac.cn

%\eit

\begin{abstract}
This paper studies the matrix Moment-SOS hierarchy for
solving polynomial matrix optimization.
Our first result is to show the   finite convergence of this hierarchy,
if the nondegeneracy condition, strict complementarity condition and
second order sufficient condition hold at every minimizer,
under the  \rev{Archimedean property}.
A useful criterion for detecting the finite convergence is the flat truncation.
Our second result is to show that every minimizer of the moment relaxation
must have a flat truncation when the relaxation order is big enough,
under the above mentioned optimality conditions.
These results give connections between
nonlinear semidefinite optimization theory
and Moment-SOS methods for solving polynomial matrix optimization.

\end{abstract}

\section{Introduction}

We consider the polynomial matrix optimization (PMO)
\be  \label{nsdp}
\left\{ \baray{rl}
\min & f(x)  \\
\st &  	G(x)\succeq 0, \\
\earay \right.
\ee
where $f(x)$ is a polynomial in $x:=(x_1,\dots,x_n)$
with real coefficients and $G(x)$ is an $m\times m$ symmetric  polynomial matrix in $x$.
%\ter{Let $S(G)$ be the feasible set of \reff{nsdp}}.
The PMO contains a broad class of important optimization problems
\cite{dehe,glss2,glss,HHP15,hdlb,hdlb2,HDLJ,klms,schhol,yzgf}.
%
%we refer to the survey \cite{cgcon} for more details.
%

A standard approach for solving \reff{nsdp} globally is the \rev{matrix}
Moment-sum-of-squares (SOS) hierarchy introduced in \cite{hdlb,schhol}.
\rev{It consists of a sequence of nested semidefinite programming (SDP) relaxations, whose definition requires introducing some notation.}

A polynomial $p$ is said to be an SOS if
$p=p_1^2+\dots+p_s^2$ for $p_1,\dots,p_s \in \mathbb{R}[x]$ \rev{and  some $s \in \mathbb{N}$.}
Denote by $\Sigma[x]$ the set of all SOS polynomials.
The quadratic module of scalar polynomials generated
by $G(x)$ is the set 
\rev{\be \label{qm1}
	\qmod{G}  \coloneqq  \Big \{
	\sigma+\sum_{t=1}^r v_t^TGv_t \mid
	\sigma \in \Sigma[x],~v_t \in \mathbb{R}[x]^m,~r\in \N
	\Big \} \subseteq \mR[x].
	\ee}
\rev{The quadratic module $\qmod{G}$ is said to be \rev{Archimedean}, if  there exists
	$R >0$ such that $R-\|x\|^2 \in \qmod{G}$.}	For an integer $k$, its  $2k$th degree truncation $\mathrm{QM}[G]_{2 k} \subseteq \mathbb{R}[x]$
is the set of all polynomials that can be represented as above, with $\deg(\sigma) \leq 2 k$
and $\deg(v_t^T G v_t)\leq 2k$ $(t=1,\dots,r)$.
The $k$th order SOS relaxation of \reff{nsdp} is
\be \label{sos}
\left\{\begin{array}{cl}
	\max & \gamma \\
	\text { s.t. } & f-\gamma \in \mathrm{QM}[G]_{ 2k} .
\end{array}\right.
\ee
The optimal value of \reff{sos} is the SOS bound of order $k$,
which we denote by $f_{k,sos}$.

Denote the power set
$
\mathbb{N}_{2k}^n  \coloneqq \left\{\alpha \in \mathbb{N}^n :
|\af| \leq 2k\right\}
$
(see Section~\ref{sec:notation} for details of the notation).
For a truncated multi-sequence $y \in \re^{ \mathbb{N}_{2k}^n }$,
it defines the Riesz functional
$\mathscr{L}_{y}$ acting on \rev{$\mathbb{R}[x]_{ 2k}$} as
\begin{equation} \nonumber
	\mathscr{L}_{y}(\sum_{\alpha \in \mathbb{N}_{2k}^n} p_{\alpha} x^{\alpha})
	\,  \coloneqq  \, \sum_{\alpha \in \mathbb{N}_{2k}^n} p_{\alpha} y_{\alpha}.
\end{equation}
The dual optimization of \reff{sos} is the $k$th order moment relaxation
\be  \label{mom}
\left\{ \baray{cl}
\min\limits_{y \in \mathbb{R}^{ \mathbb{N}_{2k}^n }} &  \mathscr{L}_{y}(f)  \\
\st  &  \mathscr{L}_{y}(1)=1,  \\
&  \mathscr{L}_{y}(p)\geq 0,~\forall p\in \qmod{G}_{2k}.
\earay \right.
\ee
We denote by $f_{k,mom}$  the optimal value of \reff{mom}, which is the moment bound of order $k$. 
For $k=1,2, \ldots$, the sequence of relaxations \reff{sos}-\reff{mom}
is called the matrix Moment-SOS hierarchy.
We refer to Section~\ref{sec:momrel} for more details about the above notation.

\rev{Let $S(G)=\left\{x \in \mathbb{R}^n: G(x) \succeq 0\right\}$ denote the feasible set of $(\ref{nsdp})$}, and
let $f_{\min}$ denote the optimal value of \reff{nsdp}.
Then, it holds that $f_{k,sos}\leq f_{k,mom}$ by weak duality and the following monotonicity relation
\[
\cdots \le f_{k,sos} \leq f_{k+1,sos} \leq \cdots \leq f_{\min} .
\]
The hierarchy \reff{sos}-\reff{mom} is said to have  {\it asymptotic convergence} if
$f_{k,sos}  \rightarrow  f_{\min}$ as $k\rightarrow \infty$, and it has {\it finite convergence}
if there \rev{exists} an order $N>0$ such that $f_{k,sos}=f_{\min}$ for all $k \ge N$, \rev{which of course implies $f_{k,mom}=f_{\min }$.}

When the polynomial matrix $G(x)$ is diagonal, the problem \reff{nsdp}
reduces to the  \rev{scalar} polynomial optimization problem
and the hierarchy of \reff{sos}-\reff{mom} becomes
the Moment-SOS hierarchy introduced by Lasserre \cite{Las01}.
For classical polynomial optimization, there are rich convergence results
for the Moment-SOS hierarchy. Under the \rev{Archimedean property},
%(i.e., there exists an positive  integer $N$ such that $N-\|x\|^2\in \qmod{G}  $),
Lasserre \cite{Las01} proved the asymptotic convergence for the Moment-SOS hierarchy.
Later, Nie \cite{nieopcd} proved that
the Moment-SOS hierarchy has finite convergence
if further some optimality conditions
(i.e., the linear independence constraint qualification,
strict complementarity and second order sufficient conditions)
hold at every minimizer\footnote{Throughout this paper, a minimizer means a global minimizer,
	unless its meaning is otherwise specified.}.
\rev{Moreover}, these optimality conditions are shown to hold
for generic polynomials \cite{huang2023,nieopcd}.
\rev{When the polynomial constraints contain some polynomial equations with a finite set of common solutions},
the Moment-SOS hierarchy  also  has finite convergence,
as shown in \cite{LLR08,Lau07,Nie13}.
For convex polynomial optimization,
the Moment-SOS hierarchy has finite convergence under the strict convexity
or sos-convexity assumption \cite{dKlLau11,Las09}.
%For homogeneous polynomial optimization with sphere constraints,
%the finite convergence holds under certain assumptions \cite{hl}.
For polynomial optimization \rev{where the  feasible set is  unbounded}
(the \rev{Archimedean property} fails to hold for such cases),
the Moment-SOS hierarchy based on homogenization
also has finite convergence under similar
optimality condition assumptions \cite{hny,hny3,hny2}.
%
%We refer to \cite{mar06,mar09} for more results on finite convergence.
%

Despite the rich work for scalar  polynomial optimization,
there is relatively less work for the matrix Moment-SOS hierarchy.
Under the \rev{Archimedean property}, this  hierarchy
has asymptotic convergence, as shown in \cite{hdlb,schhol}.
However, very little is known on the finite convergence  of the matrix Moment-SOS hierarchy
of \reff{sos}-\reff{mom}, to the best of the authors' knowledge.

This paper studies the finite convergence theory of
the  Moment-SOS hierarchy of \reff{sos}-\reff{mom},
for solving the  polynomial matrix optimization \reff{nsdp}.
To this end, we first introduce the nondegeneracy condition
in nonlinear semidefinite optimization.

%\subsection{Nondegeneracy condition in nonlinear semidefinite optimization}
\subsection{Nondegeneracy condition}
\label{seccq}

For an $m$-by-$m$ symmetric polynomial matrix
%matrix-valued function
$
G  = ( G_{st}(x))_{s,t=1,\ldots,m},
$
the derivative of $G$ at $x$ is the linear mapping
$\nabla G(x): \mathbb{R}^n \rightarrow \mathcal{S}^m$
such that
\[
d \, \coloneqq \, \left(d_1, \ldots, d_n\right) \, \mapsto \,
\nabla G(x)[d] \, \coloneqq \, \sum_{i=1}^n d_i \nabla_{x_i} G(x),
\]
where $\mathcal{S}^{m}$ denotes
the set of all $m$-by-$m$ real symmetric matrices and
$\nabla_{x_i} G(x) $ denotes the partial derivative of $G$
with respect to $x_i$, i.e.,
\[
\nabla_{x_i} G(x) \, = \, \left(\frac{\partial G_{st}}{\partial x_i}\right)_{s,t=1,\dots,m}.
\]
The adjoint of $\nabla G(x)$ is the linear mapping
$\nabla G(x)^*: \mathcal{S}^m \rightarrow \mathbb{R}^n$
such that for $(d, X) \in \mathbb{R}^n \times \mathcal{S}^m$, the following holds
\[
\left\langle d, \nabla G(x)^*[X]\right\rangle \, = \,
\langle\nabla G(x)[d], X\rangle .
\]
Equivalently, for all $X \in \mathcal{S}^m$, we have the expression
\[
\nabla G(x)^*[X] = \bbm \langle\nabla_{x_1} G(x), X\rangle
& \cdots  & \langle\nabla_{x_n} G(x), X\rangle \ebm^T.
\]
%\left[\begin{array}{c}
	%	\left\langle\nabla_{x_1} G(x), M\right\rangle \\
	%	\vdots \\
	%	\left\langle\nabla_{x_n} G(x), M\right\rangle
	%\end{array}\right]=\sum_{i, j=1}^m M_{i j} \nabla G_{i j}(x)
	% where $\nabla G_{i j}(x)$ denotes the gradient of the $(i, j)$-th entry of $G$ as a function of $x$.
	% Similarly, we shall denote the gradient of any real-valued function $F: \mathbb{R}^n \rightarrow \mathbb{R}$ at a point $x \in \mathbb{R}^n$ by $\nabla F(x)$.
	At a point $u \in \mR^n$ with $G(u) \succeq 0$,
	the (Bouligand) tangent cone to the positive semidefinite cone of all
	$m$-by-$m$ symmetric matrices $\mathcal{S}_{+}^m$
	at $G(u)$ is
	\[
	T_{\mathcal{S}_{+}^m}(G(u))  = \left\{N \in \mathcal{S}^m: E^T N E \succeq 0\right\},
	\]
	where $E$ is a matrix whose column vectors form a basis of the null space
	(i.e., the kernel)
	of $G(u)$.  Thus, the lineality space at $u$ can be characterized as
	\[
	\operatorname{lin}\left(T_{\mathcal{S}_{+}^m}(G(u))\right) \, =  \,
	\left\{N \in \mathcal{S}^m: E^T N E=0\right\},
	\]
	which is the largest subspace contained in $T_{\mathcal{S}_{+}^m}(G(u))$.
	
	%		\subsubsection*{Nondegeneracy condition}
	%	\label{ndcon}
	
	Now we can introduce the nondegeneracy condition (NDC).
	Let $u$ be a feasible point of \reff{nsdp}.	
	The {\it nondegeneracy condition} (NDC),
	also called {\it transversality constraint qualification} in some literature \rev{(see \cite{shap})},
	is said to hold at $u$ if
	\be \label{CQ}
	\operatorname{Im} \nabla G(u)+\operatorname{lin}
	\left(T_{\mathcal{S}_{+}^m}(G(u))\right)
	\, = \, \mathcal{S}^m . \tag{NDC}
	\ee
	In the above, $\operatorname{Im} \nabla G(u)$
	denotes the image of the mapping $\nabla G(u)$,
	i.e., it is the linear span of matrices $\nabla_{x_i} G(u)$
	for $i =1, \ldots, n$. Under the NDC, we can derive first and second order
	optimality conditions for $u$ to be a local minimizer of \reff{nsdp}.
	More details can be found in Section~\ref{opt:nsdp}.

	The NDC is generally viewed as an analogue of the classical
	linear independence constraint qualification condition (LICQC)
	in nonlinear programming, as it implies the uniqueness of
	the Lagrange multiplier matrix (see Section \ref{firstopt}).
	\rev{However, the NDC is not equivalent to the LICQC when $G(x)$ is diagonal with diagonal entries $g_1(x),\cdots, g_m(x)\in \mR[x]$. The LICQC is said to hold at a feasible $u\in \mR^n$ for polynomial inequalities $g_1(x)\geq 0,\cdots, g_m(x)\geq 0$ if the gradients of the active constraints are linearly independent.}
	For instance, consider the polynomial matrix
	\[
	G(x)  = \left[\begin{array}{ll}
		x_1  & 0  \\
		0& x_2  \\
	\end{array}\right].
	\]
	The matrix inequality $G(x) \succeq 0$  can be equivalently given as $x_1\geq0$, $x_2\geq 0$.
	Clearly, the LICQC holds for constraints $x_1\geq0$, $x_2\geq 0$ at $(0,0)$,
	since the gradient vectors $(1,0)$, $(0,1)$ are linearly independent.
	However, the nondegeneracy condition fails
	to hold at $(0,0)$, because
	%$$
	%\operatorname{Im} \nabla G(x)+\operatorname{lin}\left(T_{\mathcal{S}_{+}^2}(G(x))\right)
	%\subseteq \{S\in \mathcal{S}^2\mid S_{12}=0 \}\neq \mathcal{S}^2,
	%$$
	%by direct computations.
	\[
	\operatorname{Im} \nabla G(0)+\operatorname{lin}\left(T_{\mathcal{S}_{+}^2}(G(0))\right)
	\, = \, \Big \{ \bbm d_1 & 0 \\ 0 & d_2 \ebm : d_1, d_2 \in \re \Big \}
	\neq \mathcal{S}^2 .
	\]

	In general, when $G$ is a diagonal matrix with diagonal entries $g_1,\ldots,g_m\in \mR[x]$,
	the matrix inequality $G(x)\succeq 0$ is equivalent to 
	\[
	g_1(x) \geq 0, \ldots, g_m(x) \geq 0.
	\]
	However, for a point $u$ with $G(u)\succeq 0$, the NDC always fails at $u$
	if there exist $i_1<i_2$ such that $g_{i_1}(u)=g_{i_2}(u)=0$,
	regardless of linear independence of
	$\nabla g_1(u), \ldots, \nabla g_m(u)$.
	This is because
	$$
	\operatorname{Im} \nabla G(u)+\operatorname{lin}\left(T_{\mathcal{S}_{+}^m}(G(u))\right)
	\subseteq \{X \in \mathcal{S}^m\mid X_{i_1i_2}=0 \}\neq \mathcal{S}^m.
	$$
	In the above, $X_{i_1i_2}$ denotes the $(i_1, i_2)$th entry of $X$.
	Hence,  if the NDC holds at $u$, there exists at most one
	$1\leq i \leq m$ such that $g_i(u)=0$. For this case, the NDC
	is equivalent to the LICQC.

	%\subsection{Main results}
	\subsection{Contributions}

	Our first major result is that the hierarchy of \reff{sos}-\reff{mom}
	has finite convergence if the nondegeneracy condition,
	strict complementarity condition and second order sufficient condition hold
	at every minimizer of \reff{nsdp}. We refer to Sections~\ref{seccq} and \ref{opt:nsdp}
	for definitions of these  conditions.

	%The following is our finite convergence result.

	\begin{thm}  \label{mianthm1}
		Suppose the quadratic module $\qmod{G}$ is \rev{Archimedean}.
		If the nondegeneracy condition, strict complementarity condition
		and second order sufficient condition hold at every minimizer of \reff{nsdp},
		then the matrix Moment-SOS hierarchy of \reff{sos}-\reff{mom} has finite convergence,
		i.e.,  there exists $N>0$ such that
		$f_{k,sos}=f_{k,mom}=f_{\min}$ for all $k \ge N$.
		Furthermore, we have $f-f_{\min}\in \qmod{G}$.
	\end{thm}

	The proof of Theorem \ref{mianthm1} uses  a result
	(see Theorem \ref{scal:bhc}) by Scheiderer on sum-of-squares representations
	for nonnegative polynomials that have only
	finitely many zeros \cite{Sch05}, which is also known as a version of
	the local-global principle for nonnegative scalar polynomials.
	In addition, it was shown in \cite{dgs} that  almost all linear perturbations
	(i.e., except for a set of zero measure in the input space)
	of a given polynomial matrix optimization problem generate perturbed problems
	in which the nondegeneracy condition,
	strict complementarity condition
	and second order sufficient condition hold at every minimizer.
	Hence, the optimality assumptions in Theorem \ref{mianthm1} hold generically.

	%We would like to remark that the nondegeneracy condition,
	%strict complementarity condition and second order sufficient condition hold
	%at every local minimizer of \reff{nsdp}, under generic assumptions,
	%which is shown in \cite{dgs}.

	There are very different algebraic and geometric properties
	between polynomial matrix inequalities and scalar polynomial inequalities.
	%there are essential differences between the matrix case and the scalar case,
	%and the matrix case is more difficult than the scalar case.
	We refer to Section~\ref{proofdif} for a detailed analysis and
	refer to Section~\ref{sec:proofs} for the  proof outline
	of Theorem~\ref{mianthm1}.
	Here is a short summary about the \rev{differences}:
	
	\bit
	
	%\item   When $G$ is diagonal, Theorem \ref{mianthm1}  does not imply
	%the finite convergence result for the scalar polynomial inequality case shown in \cite{nieopcd}.
	%This is because the NDC in nonlinear semidefinite optimization
	%has no intrinsic relation with the LICQC for its natural scalar descriptions,
	%as we have seen in the above.

	\item  The optimality conditions for nonlinear semidefinite optimization
	have geometric properties that are very different from those
	in nonlinear scalar constrained optimization.
	When $G$ is diagonal, the NDC for \reff{nsdp}
	is not equivalent to the LICQC for the corresponding scalar constrained optimization,
	as we have seen in Section \ref{seccq}.
	%For this reason,
	%the finite convergence theory for scalar cases in \cite{nieopcd}
	%cannot be applied to prove Theorem \ref{mianthm1}.
	Moreover, the existence of the ``$H$-term" in the second order sufficient condition
	\reff{sosc} (see Section~\ref{sec:sosc}) is another difficulty
	for proving the finite convergence. In contrast,
	this term does not appear in the  second order sufficient condition
	for nonlinear scalar constrained optimization.

	\item For a general  polynomial matrix $G$,
	the feasible set $S(G)$ of \reff{nsdp} usually has singular points on its boundary
	but the NDC holds. This is the case even if $G$ is linear.
	The reason is that at a minimizer $u\in \mR^n$, the rank of $G(u)$
	is usually less than $m-1$ (see \cite{NRS10}). For instance, consider
	\begin{equation} \label{intexm}
		G(x) = \left[\begin{array}{ccc}
			1   &~ x_1  &~ x_2 \\
			x_1 &~ 1    &~ x_3 \\
			x_2 &~ x_3  &~ 1
		\end{array}\right] \succeq 0,
	\end{equation}
	which gives the three-dimensional elliptope $\mathcal{E}_3$.
	The algebraic boundary of $\mathcal{E}_3$ is the cubic surface  defined by
	\[
	\det\, G(x) \, = \, (1-x_1^2)(1-x_2^2)-(x_3-x_1x_2)^2=0.
	\]
	One can see that  $u=(1,1,1)$
	is an isolated singular point on the surface $\det\, G(x) = 0$
	since $\nabla \det\, G(u)=0$. In contrast, the NDC holds at $u$.
	For such case, the finite convergence result for scalar polynomial inequalities
	developed in \cite{nieopcd} does not apply,
	since nonsingularity is a basic assumption there.
	However, Theorem~\ref{mianthm1} can be applied to
	show the finite convergence, under the given optimality conditions.

	\item A polynomial matrix inequality can be equivalently given
	by scalar polynomial inequalities, e.g., by all its principal minors.
	However, the scalarization usually has the singularity (i.e., the LICQC fails)
	and the scalar describing  polynomials (e.g., determinants)
	usually do not belong to the quadratic module $\qmod{G}$.
	This prevents us from using the existing finite convergence theory
	for scalar polynomial inequalities
	to get that for  polynomial matrix inequalities,
	which is given in Theorem~\ref{mianthm1}.

	\item  The proof of finite convergence for scalar polynomial inequalities
	relies on the boundary hessian condition (BHC) by Marshall \cite{mar06}.
	%%which implies the conditions appearing in Theorem \ref{scal:bhc}.
	%However, a matrix version of the boundary Hessian condition and
	%the corresponding local-global principle
	%for polynomial matrices are not available.
	However, the BHC usually does not hold for polynomial matrix inequalities,
	even for linear $G$, due to the singularity (see the above).

	\item For a  polynomial matrix $G$, the quadratic module $\qmod{G}$
	is usually not finitely generated, as shown in \cite{cj}.
	However, most earlier existing work about finite convergence
	relies on the assumption that the quadratic module is finitely generated.
	%We refer to Section~\ref{sec:qm} for the discussion about this issue.

	\eit

	A convenient criterion for detecting finite convergence
	and for extracting minimizers from moment relaxations
	is the flat extension or truncation \cite{CF05,HenLas05,hdlb,Lau05,nie2013certifying}
	(see Section~\ref{sec:momrel} for details).
	Our second major result is that under the same optimality conditions,
	every minimizer of the moment relaxation \reff{mom} must have a  flat truncation when the relaxation order is sufficiently large.
	This conclusion is similar to the scalar constrained case as in \rev{\cite{hny2,nie2013certifying}.}
	An ideal $I\subseteq \mR[x]$ is said to be zero dimensional if its complex variety
	%$\{x\in \mC^n:  h(x)=0, \,\forall h\in I \}$
	is a finite set, i.e., the polynomials in $I$ have only finitely many \rev{complex}  common  zeros.
	For a quadratic module $M\subseteq \mR[x]$, its support is defined to be the set
	\[
	\supp{M} \, \coloneqq \,   M\cap (-M),
	\]
	which is an ideal in $\mR[x]$;
	see \cite{marshall2008positive,Sch09}. For a polynomial tuple $h$,
	$\ideal{h}$ denotes the ideal generated by $h$.
	
	\begin{thm}  \label{mianthm2}
		Suppose the quadratic module $\qmod{G}$  is \rev{Archimedean}.
		If  the nondegeneracy condition, strict complementarity condition
		and second order sufficient condition hold at every minimizer of \reff{nsdp},
		then we have:
		\bit
		\item[(i)] The support $\supp{M }$ is zero dimensional for the quadratic module
		\[
		M = \qmod{G}+\ideal{f-f_{\min}}  .
		\]
		%the ideal  $(\qmod{G}+\ideal{f-f_{\min}}) \cap-(\qmod{G}+\ideal{f-f_{\min}})$ has dimension $0$.

		\item[(ii)] Every minimizer of the moment relaxation~\reff{mom}
		must have a flat truncation,
		when the relaxation order $k$ is sufficiently large.
		\eit	
	\end{thm}
	
	We remark that Theorem~\ref{mianthm2}
	makes the same assumptions as in Theorem \ref{mianthm1},
	so it also implies that the Moment-SOS hierarchy of \reff{sos}-\reff{mom} has finite convergence,
	i.e., $f_{k,sos}=f_{k,mom}=f_{\min}$ for all $k$ big enough. The main motivation of
	Theorem~\ref{mianthm2} is to show how to detect finite convergence
	by checking the flat truncation condition.
	When it holds, we not only certify  finite convergence
	but also extract minimizers for \reff{nsdp}.

	\bigskip

	The paper is organized as follows.
	Section~\ref{sc:pre} reviews some basic results
	about optimality conditions for nonlinear semidefinite optimization
	and real algebraic geometry.  Section~\ref{sec:momrel} introduces
	the matrix Moment-SOS relaxations and basic properties.
	In Section~\ref{sec:proof}, we give proofs for
	Theorems \ref{mianthm1} and \ref{mianthm2}.
	Section~\ref{sec:dis} makes conclusions.

	\section{Preliminaries}
	\label{sc:pre}
	
	\subsection{Notation}\label{sec:notation}
	The symbol $\mathbb{N}$ (resp., $\mathbb{R}$, $\cpx$) denotes the set of
	nonnegative integral (resp., real, complex) numbers.
	Let $\mathbb{R}[x]:=\mathbb{R}\left[x_1, \ldots, x_n\right]$
	be the ring of polynomials in $x:=\left(x_1, \ldots, x_n\right)$
	with real coefficients and let $\mathbb{R}[x]_d$
	be the ring of polynomials with degrees $\leq d$.
	For $\ell\in \mathbb{N}$,  denote by  $\mathcal{S}^{\ell}$
	(resp., $\mathcal{S}^{\ell} \mathbb{R}[x]$)
	the set of all $\ell$-by-$\ell$  real symmetric matrices
	(resp., $\ell$-by-$\ell$ symmetric polynomial matrices).
	For $\af = (\af_1, \ldots, \af_n) \in \N^n$,
	denote $|\af| \coloneqq \af_1 + \cdots + \af_n$.
	For an integer $d$, denote
	\[
	\sigma(d)  \, \coloneqq \,   d(d+1)/2, \quad
	\mathbb{N}_d^n \, \coloneqq \,  \left\{\alpha \in \mathbb{N}^n \mid
	|\af| \leq d\right\}.
	\]
	For a polynomial $p, \operatorname{deg}(p)$ denotes its total degree.
	For a real number $t,\lceil t\rceil$ denotes the smallest integer that
	is greater than or equal to $t$.
	The symbol $e_i$ represents the unit vector with all entries equal to zero except for the $i$-th entry, which is one.
	The symbol $I_{\ell}$ stands for the $\ell$-by-$\ell$ identity matrix.
	For a matrix $A$, $A^{T}$ denotes its transpose. The kernel (i.e., null space) of a matrix $A$ is denoted as $\operatorname{ker} A$.
	For two matrices $W$, $Y$, their inner product is
	$\langle W, Y\rangle =\operatorname{tr}(WY^T)$.
	For a symmetric matrix $X$, the inequality $X \succeq 0$ means that
	$X$ is positive semidefinite (psd). Denote by $\mathcal{S}_{+}^m$
	the cone of all $m$-by-$m$ symmetric psd matrices.
	
	For a smooth function $p$ in $x$,
	its gradient and the Hessian with respect to $x$
	are denoted by $\nabla p$ and $\nabla^2 p$.  If $z$ is a subvector of $x$,
	then $\nabla_{z}p$  denotes its gradient
	with respect to $z$. If $z,w$ are two subvectors of $x$,
	say, $z=(x_{i_1},\dots,x_{i_{\ell_1}})$, $w=(x_{j_1},\dots,\rev{x_{j_{\ell_2}}})$,
	then we denote the Hessian of $p$
	with respect to $z,w$ as $\nabla_{z,w}^2 p$, i.e.,
	\[
	\nabla_{z,w}^2 p = \left(\nabla_{x_{i_t}\rev{x_{j_k}}}^2p
	\right)_{ \substack{ t=1,\dots,\ell_1, \\ k=1,\dots,\ell_2 }  }.
	\]

	\subsection{Optimality conditions for nonlinear semidefinite optimization} \label{opt:nsdp}

	In this subsection, we review basic theory for nonlinear semidefinite optimization.
	We refer to \cite{dgs,fa,shap,shuni,sunde} for more details.

	\subsubsection{First order optimality condition}\label{firstopt}	
	Consider a polynomial matrix  $G(x)\in \mathcal{S}^{m} \mathbb{R}[x]$.	Under the nondegeneracy condition \reff{CQ}, we can get
	the following KKT conditions for \reff{nsdp},
	which are shown in \cite{shap}.
	
	\begin{lem}
		Suppose $u$ is a local minimizer of \reff{nsdp}
		and the nondegeneracy condition \reff{CQ} holds at $u$,
		then there exists a unique matrix $\Lambda\in \mathcal{S}^m$ satisfying:
		\be \label{kkt}
		\left\{ \begin{array}{rcl}
			\nabla f(u)-\nabla G(u)^* \Lambda &=& 0, \\
			\Lambda \succeq 0, \quad  G(u) &\succeq& 0, \\
			\langle \Lambda,\,  G(u)\rangle &=& 0.
		\end{array}  \right.  \tag{FOOC}
		\ee
	\end{lem}
	The above is called the {\it first order optimality condition} (FOOC).
	Furthermore, if
	\be \label{SCC}
	\operatorname{rank} G(u)+\operatorname{rank} \Lambda=m,  \tag{SCC}
	\ee
	then $(u, \Lambda)$ is said to satisfy the {\it strict complementarity condition} (SCC).
	The nondegeneracy condition  implies the uniqueness of
	the Lagrange multiplier matrix $\Lambda$ in \reff{kkt}.
	To see this, suppose otherwise there exist $\Lambda_1,\Lambda_2\in \mathcal{S}^m$
	such that \reff{kkt} holds. Then, we have  $\nabla G(u)^*[\Lambda_1-\Lambda_2]=0$, which means that $\left\langle\nabla_{x_i} G(u), \Lambda_1-\Lambda_2\right\rangle=0$ for all $i=1,\dots,n$.
	So, $\Lambda_1-\Lambda_2$ belongs to the orthogonal complement of $\operatorname{Im} \nabla G(u)$.
	Since matrices $\Lambda_1,~\Lambda_2, ~G(u)$ are positive semidefinite, the condition $\langle G(u), \Lambda_i\rangle=0$ means that
	$\operatorname{Im} \Lambda_i \subseteq \operatorname{ker} G(u)$ $(i=1,2)$
	and  thus  $\Lambda_1-\Lambda_2$ belongs to the orthogonal complement of $\operatorname{lin}(T_{\mathcal{S}_{+}^n}(G(u)))$.
	Then, it follows from \reff{CQ} that $\Lambda_1=\Lambda_2$.

	%\textcolor{red}{(It may be better to delete the next sentence?)}
	%This analogy is not perfect, as it
	%does not reduce to the classical linear independence constraint qualification when the
	%matrix $G(x)$ is diagonal.

	\subsubsection{Second order  necessary and sufficient  conditions}\label{sec:sosc}
	For the Lagrange multiplier matrix $\Lambda$ in \reff{kkt},
	define the Lagrange function
	\[
	L(x, \Lambda)  \, \coloneqq \, f(x)-\langle \Lambda,  G(x)\rangle.
	\]
	In the following, we give the second order optimality conditions of \reff{nsdp}.
	For $x \in \mathbb{R}^n$, we define the following linear subspace of $\mathbb{R}^n$:
	%\ter{(In the later, $T(x)$ denotes a polynomial matrix.
		%Should we use a different notation like $\mc{N}(x)$?)}
	\be \label{tangent}
	\mc{N}(x) \, \coloneqq \,
	\left\{h=\left(h_1, \ldots, h_n\right) \in \mathbb{R}^n \Big \vert
	\sum_{i=1}^n h_i \cdot E^T \nabla_{x_i} G(x) E=0\right\},
	\ee
	where the column vectors of the matrix $E$ form a basis of
	$\operatorname{ker} G(x)$. For $x \in \mathbb{R}^n$ and $Q \in \mathcal{S}^m$,
	denote the symmetric $n$-by-$n$ matrix $H(x, Q)$,
	with entries $\rev{H(x, Q)_{i j}}~(1 \leq i, j \leq n)$ given as
	\[
	H(x, Q)_{i j} \, \coloneqq \,
	2 \langle Q, \nabla_{x_i} G(x) G(x)^{\dagger} \nabla_{x_j} G(x) \rangle.
	\]
	In the above, the superscript $^{\dagger}$ denotes the
	Moore-Penrose inverse of the matrix $G(x)$ (see \cite{shap}).
	
	The second order necessary  condition below was obtained by Shapiro \cite{shap}.
	
	\begin{lem}
		Suppose $u$ is a local minimizer of \reff{nsdp} and the nondegeneracy condition holds at $u$.
		Let $\Lambda$ be as in \reff{kkt}. Then, the {\it second order necessary condition}
		(SONC) holds at $u$,
		i.e., for every $0 \ne h \in \mc{N}(u)$, we have
		\[
		h^T\left(\nabla^2 L(u, \Lambda)+H(u, \Lambda)\right) h\geq 0. \tag{SONC}
		\]
	\end{lem}
	Furthermore, we say  the {\it second order sufficient condition}
	(SOSC) holds at $u$ if   for every $0 \ne h \in \mc{N}(u)$,
	\be \label{sosc}
	h^{T}\left(\nabla^2 L(u, \Lambda)+H(u, \Lambda)\right) h > 0.  \tag{SOSC}
	\ee
	The above \rev{matrix} $H(u, \Lambda)$ is well-known and often mentioned as the ``$H$-term"
	in nonlinear semidefinite optimization. In contrast,
	this term does not appear in the  second order sufficient condition
	for nonlinear scalar constrained optimization.	
	If the NDC, SCC, SOSC all hold at $u$, then $u$ is a strict local minimizer of \reff{nsdp}.
	In the following, we summarize all these facts.
	
	\begin{thm}
		Let $u$ be feasible for \reff{nsdp}. Then, we have that
		\bit
		\item[(i)] If $u$ is a local minimizer and the NDC holds at $u$, then the FOOC and SONC hold.
		
		\item[(ii)] If the NDC, SCC and SOSC hold at $u$, then $u$
		is a strict local minimizer of \reff{nsdp}.
		\eit
	\end{thm}
	
	We remark that when $G$ is diagonal, the NDC for \reff{nsdp}
	is not equivalent to the LICQC for the scalar constrained optimization,
	as discussed in Section \ref{seccq}.
	Thus, the above theorem is not equivalent to its classical analogue in the scalar case when $G$ is diagonal.

	\subsection{Some basics in real algebraic geometry}
	\label{sec:qm}

	In this subsection, we review some basics in polynomial optimization and real algebraic geometry.
	We refer to \cite{bcr,cj,HDLJ,kism,LasBk15,Lau09,niebook,Sch09,sk09}
	for more detailed introductions.
	
	For a degree $d\in \mathbb{N}$, denote the monomial vector
	\[
	[x]_{d}:=\left[\begin{array}{llllllll}
		1 & x_1 & \cdots & x_n & x_1^2 & x_1 x_2 & \cdots & x_n^{d}
	\end{array}\right]^T.
	\]
	A subset $I  \subseteq \mathbb{R}[x]$ is called an ideal of $\re[x]$
	if $I \cdot \mathbb{R}[x] \subseteq I$, $I+I \subseteq I$.
	Its real variety is the set of real common roots to all polynomials in $I$, which is defined as
	\begin{equation*}
		V_{\mR}(I)=\{x \in \mR^n \mid p(x)=0, ~ \forall p\in I\}.
	\end{equation*}
	Similarly, its complex variety is defined as
	\begin{equation*}
		V_{\mC}(I)=\{x \in \mC^n \mid p(x)=0, ~ \forall p\in I\}.
	\end{equation*}
	For a polynomial tuple $h  \coloneqq (h_1,\dots, h_s)$,
	$\ideal{h}$ denotes the ideal generated by $h$, i.e.,
	\begin{equation*}
		\ideal{h} \,= \, h_1 \cdot \mathbb{R}[x]+\cdots+h_s \cdot \mathbb{R}[x].
	\end{equation*}
	A polynomial $p$ is said to be a sum of squares (SOS) if
	\[
	p=p_1^2+\dots+p_t^2\,\,\quad  \text{for} \,\,\,p_1,\dots,p_t \in \mathbb{R}[x]\,\,\,  \rev{\text{and} \,\,\,t\in \mathbb{N}}.
	\]
	The set of all SOS polynomials in $x$ is denoted as $\Sigma[x]$.
	For a  degree $k$, denote the truncation
	\[
	\Sigma[x]_{k} \, \coloneqq  \, \Sigma[x] \cap  \mathbb{R}[x]_{k} .
	\]
	For a polynomial tuple $g=(g_1,\dots,g_m)$,
	the  quadratic module generated by $g$ is
	\be
	\qmod{g} \,  \coloneqq  \,  \Sigma[x]+ g_1 \cdot \Sigma[x]+\cdots+ g_m \cdot \Sigma[x].
	\ee
	The $k$th degree truncation of $\qmod{g}$ is
	\be
	\qmod{g}_{k} = \Sigma[x]_{k}+ g_1 \cdot \Sigma[x]_{k-\deg(g_1)}+\cdots
	+ g_m \cdot \Sigma[x]_{k-\deg(g_m)}.
	\ee
	A set $M\subseteq \mR[x]$ is called a quadratic module if it satisfies
	\[
	1\in M, \quad M+M\subseteq M, \quad \Sigma[x]\cdot M \subseteq M .
	\]
	The quadratic module $M$ is said to be finitely generated if there exists a polynomial tuple $g=(g_1,\dots,g_m)$ such that $M=\qmod{g}$.
	
	The sum $\ideal{h}+\qmod{g}$ is said to be \rev{Archimedean}, if  there exists
	$R >0$ such that $R-\|x\|^2 \in \ideal{h}+\qmod{g}$.
	If it is \rev{Archimedean}, then the set
	\[
	T \coloneqq \left\{x \in \mathbb{R}^{n} \mid h(x)=0, g(x) \geq 0\right\}
	\]
	must be compact. Clearly, if $p \in \ideal{h}+\qmod{g}$, then $p \ge 0$ on $T$
	while the converse is not always true. However,
	if $p$ is positive on $T$ and $\ideal{h}+\qmod{g}$ is \rev{Archimedean}, we have $p \in \ideal{h}+\qmod{g}$.
	This  is referred to as Putinar's Positivstellensatz \cite{putinar1993positive}.
	
	In the following, we define SOS polynomial matrices since \rev{they} will be used later in the proofs. For a positive integer $\ell$, the cone of $\ell$-by-$\ell$ SOS  polynomial matrices is the set
	\[
	\mathcal{S}^{\ell}\Sigma[x] \, \coloneqq \,  \Big \{
	\sum_{i=1}^{r} P_{i}^T P_{i}   \mid
	r \in \mathbb{N}, ~ P_{i} \in \mathbb{R}[x]^{\ell \times \ell}\Big \}.
	\]
	For an $m$-by-$m$ symmetric  polynomial matrix $G(x)$,
	its quadratic module of \rev{$\ell$-by-$\ell$}  polynomial matrices is
	\[
	\mathcal{S}^{\ell}\mathrm{QM}[G] \, \coloneqq \, \Big \{
	Q+\sum_{j=1}^r P_{j}^T G P_{j}   \Big \vert
	r \in \mathbb{N}, ~ Q\in\mathcal{S}^{\ell}\Sigma[x],~P_{j} \in \mathbb{R}[x]^{m \times \ell}
	\Big \} .
	\]
	Clearly, $\mathcal{S}^{\ell}\Sigma[x] \subseteq \mathcal{S}^{\ell}\mathrm{QM}[G]$.  \rev{Observe that when $\ell=1$, we recover the quadratic module $Q M[G]$ as in \reff{qm1}}. The quadratic module $\mathcal{S}^{\ell}\mathrm{QM}[G]$ is said to be \rev{Archimedean}
	if there exists  $R>0$ such that
	$
	(R-\|x\|^2)\cdot I_{\ell} \in \mathcal{S}^{\ell}\mathrm{QM}[G].
	$
	Putinar's Positivstellensatz has been generalized to the matrix case by Scherer and Hol  \cite{schhol}.

	\begin{thm}[\cite{schhol}]
		Suppose   the quadratic module $\mathcal{S}^{\ell}\mathrm{QM}[G]$ is \rev{Archimedean}, and $F(x)\in \mathcal{S}^{\ell}\mR[x]$. If $F\succ0$ on  the set $\{x\in\mR^n\mid G(x)\succeq 0\}$, then we have  $F \in \mathcal{S}^{\ell}\mathrm{QM}[G]$.
	\end{thm}

	%
	%\subsection{Some technical propositions}
	%
	\bigskip
	In the following, we give a result by Scheiderer on representing nonnegative polynomials with finitely many zeros (also known as a version of the  local-global principle for nonnegative scalar polynomials in the literature), which plays a crucial role in the proof. For the polynomial tuple
	$
	g=(g_1, g_2,\dots, g_m),
	$
	denote the set
	\[
	S(g)=\{x\in \mR^n:g_1(x)\geq 0, g_2(x)\geq 0,\dots, g_m(x) \geq 0\}.
	\]
	%Let $M \subseteq \re[x]$ be the quadratic module generated by $g$.
	For a point $u = (u_1, \ldots, u_n) \in \mR^n$, denote by $\widehat{\mR[x]}_u$
	the completion of the local ring $\re[x]_u$,
	which is the ring of rational functions $p(x)/q(x)$ with $q(u) \ne 0$.
	In fact, the completion $\widehat{\mR[x]}_u$ is the ring
	of formal power series in $x_1-u_1, \ldots, x_n-u_n$. 
	Equivalently,
	\[
	\widehat{\mR[x]}_u = \mathbb{R}[[x_1-u_1, \ldots, x_n-u_n]].
	\]
	Since every polynomial can be viewed as a power series with zero coefficients beyond a certain degree,  \rev{this} induces the natural map $\psi$ from the polynomial ring $\mathbb{R}[x]$ to  $\widehat{\mR[x]}_u$.
	We denote by $\widehat{\qmod{g}}_u$
	the quadratic module in $\widehat{\mathbb{R}[x]}_u$ generated by the image of $\qmod{g}$ under the natural map $\psi$.

	\begin{thm}[Proposition 3.4, \cite{Sch05}] \label{scal:bhc}
		Assume that  $\qmod{g}$ is \rev{Archimedean}, $f \in \mR[x]$ is nonnegative on $S(g)$
		and has only finitely many zeros in $S(g)$.
		\rev{If  $f \in \widehat{\qmod{g}}_u$
			for every $u \in V_{\mR}(f) \cap S(g)$} and at least one of the following two conditions is satisfied:
		\bit
		
		\item[(i)] $\operatorname{dim} V_{\mR}(f) \leq 1$;
		
		\item[(ii)] \rev{for every $u \in V_{\mR}(f)  \cap S(g)$, there are a neighborhood $U$
			of $u$ in $\mR^n$ and an element $a \in \qmod{g}$
			such that}
		\[
		\{x\in \mR^n: a(x) \geq 0\} \cap V_{\mR}(f) \cap U \subset S(g),
		\]
		\eit
		then $f \in \qmod{g}$.
		Furthermore, the ideal $(\qmod{g}+\ideal{f})\cap -(\qmod{g}+\ideal{f})$ is zero dimensional.
	\end{thm}

	We remark that although  the fact   that the ideal $(\qmod{g}+\ideal{f})\cap -(\qmod{g}+\ideal{f})$
	has dimension 0 is not stated explicitly  in \cite[Proposition 3.4]{Sch05},
	the proof of \cite[Proposition 3.4]{Sch05} essentially shows this.

	\section{The matrix Moment-SOS hierarchy}
	\label{sec:momrel}
	
	In this section, we briefly introduce the matrix Moment-SOS hierarchy of
	semidefinite relaxations for solving polynomial matrix  optimization.
	We refer to \cite{hdlb,niebook,schhol} for the relevant work.
	%The cone of $\ell$-by-$\ell$ SOS  polynomial matrices is
	%\[
	%\mathcal{S}^{\ell}\Sigma[x] \, \coloneqq \,  \Big \{
	%\sum_{i=1}^{k} P_{i}^T P_{i}   \mid
	%k \in \mathbb{N}, ~ P_{i} \in \mathbb{R}[x]^{\ell \times \ell}\Big \}.
	%\]
	%For the $m$-by-$m$ symmetric  polynomial matrix $G(x)$,
	%its quadratic module of $\ell$-by- $\ell$  polynomial matrices is
	%\[
	%\mathcal{S}^{\ell}\mathrm{QM}[G] \, \coloneqq \, \Big \{
	%Q+\sum_{j=1}^k P_{j}^T G P_{j}   \Big \vert
	%k \in \mathbb{N}, ~ Q\in\mathcal{S}^{\ell}\Sigma[x],~P_{j} \in \mathbb{R}[x]^{m \times \ell}
	%\Big \} .
	%\]
	%When $\ell=1$, the quadratic module of $\ell$-by- $\ell$  polynomial matrices  $\mathcal{S}^{\ell}\mathrm{QM}[G]$ reduces to the $1$-by-$1$
	%quadratic module generated by $G(x)$, for which we denote by $\qmod{G}$, i.e.,
	%\be \nonumber
	%\qmod{G}  = \Big \{
	%\sigma+\sum_{t=1}^\rev{r} v_t^TGv_t \mid   \sigma \in \Sigma[x],~v_i \in \mathbb{R}[x]^m,~\rev{r\in \N}
	%\Big \}.
	%\ee
	%Note that $\qmod{G}$
	%is a convex cone.
	% $$
	% \mathrm{QM}[\mathcal{G}]:=\mathrm{QM}[\mathcal{G}]^{1 },\,\,\Sigma[x]=\Sigma[x]^{1} .
	% $$
	%The quadratic module $\mathrm{QM}[G]$ is said to be archimedean
	%if there exists a scalar $N$ such that
	%$
	%N-\|x\|^2 \in \mathrm{QM}[G].
	%$
	%The feasible set $K$ is given as
	%\[
	%K  = \left\{x \in \mathbb{R}^n:  G(x) \succeq 0  \right\} .
	%\]
	%If $\mathrm{QM}[G]$ is archimedean, the set $K$ must be compact.
	%However, the converse is not necessarily true.
	\rev{Consider} the polynomial matrix $G\in \mathcal{S}^m\mR[x]$.
	The feasible set of \reff{nsdp} is denoted as
	\[
	S(G)  = \left\{x \in \mathbb{R}^n:  G(x) \succeq 0  \right\} .
	\]
	\rev{Define the degree parameter}
	\begin{equation}
		d_G  \coloneqq  \max  \big \{ \lceil \deg(G_{ij})/2 \rceil:
		1\leq i\leq j\leq m \big \}.
	\end{equation}
	For a degree $k$, the $2k$th degree truncation of $\qmod{G}$ is
	\be \nonumber
	\qmod{G}_{2k}:=\left\{\baray{l|l}
	\sigma+\sum_{t=1}^{r} v_t^TGv_t& \baray{l}  \sigma \in \Sigma[x],~v_t \in \mathbb{R}[x]^m,~r\in \N,\\
	\deg(\sigma)\leq 2k,~\deg(v_t^TGv_t)\leq 2k
	\earay
	\earay \right\}.
	\ee
	In the above,  the number $r$  can be upper bounded by
	$m\cdot \binom{n+k-\lceil d_G/2 \rceil}{k-\lceil d_G/2 \rceil}$,  by the Carath\'{e}odory's theorem.
	It is well-known that SOS polynomials can be modelled via SDP
	%represented by psd matrices
	and then checking polynomials in $\qmod{G}_{2k}$ can
	be done by solving semidefinite programs
	(see \cite{hdlb,schhol} or \cite[Chapter~10]{niebook}).

	For  a truncated multi-sequence (tms) $y \in \re^{ \mathbb{N}_{2k}^n }$, it determines the Riesz functional
	$\mathscr{L}_{y}$ acting on $\mathbb{R}[x]_{2k}$ as
	\begin{equation} \label{Reiz:fun}
		\mathscr{L}_{y}(\sum_{\alpha \in \mathbb{N}_{2k}^n} p_{\alpha} x^{\alpha})
		\,  \coloneqq  \, \sum_{\alpha \in \mathbb{N}_{2k}^n} p_{\alpha} y_{\alpha}.
	\end{equation}
	For a polynomial $q\in \mR[x]_{2k}$,
	the $k$th order {\em localizing matrix} of   $q$ for $y$
	is the symmetric matrix $L_{q}^{(k)}[y]$ satisfying
	\be \label{locmat:gi}
	L_{q}^{(k)}[y] \, = \, \mathscr{L}_{y}(q\cdot [x]_{t}[x]_{t}^T).
	\ee
	In the above, $t = k-\lceil \deg(q)/2 \rceil$
	%$\vec{p}$ denotes the coefficient vector of $p$ and
	and $\mathscr{L}_{y}$ is applied entry-wise to the polynomial matrix.
	In particular, for $q = 1$, the localizing matrix $L_{q}^{(k)}[y]$
	becomes the $k$th order {\it moment matrix}
	$
	M_k[y]\,\coloneqq\, L_1^{(k)}[y].
	$

	The $k$th order localizing  matrix of $G$ for
	$y$ is the block matrix
	%$$
	%L_{G}^{(k)}[y]:=\mathscr{L}_{y}\left(b_{k-d_G}(x) b_{k-d_G}(x)^T \otimes G(x)\right),
	%$$
	%where $\otimes$ stands for the Kronecker product, the Riesz functional
	% $\mathscr{L}_{y}$ defined as in \reff{Reiz:fun}
	% is applied entrywise to the polynomial matrix.
	\[
	L_{G}^{(k)}[y] \, \coloneqq \,
	( L_{ G_{ij} }^{(k)}[y] )_{1 \le i, j \le m},
	\]
	where each $L_{ G_{ij} }^{(k)}[y]$ is defined as in \reff{locmat:gi}.
	For instance, when $n=2$ and
	\[
	G(x)= \left[\begin{array}{ll}
		1-x_1x_2 &\quad x_1+x_2\\
		x_1+x_2&\quad x_1^2-x_2^2\\
	\end{array}\right],
	\]
	we have
	%\[
	%L_{G}^{(2)}[y]=\left[\begin{array}{lll}B_{00} &B_{10} & B_{01}\\
		%B_{10} & B_{20} & B_{11}\\ B_{01} & B_{11} & B_{02}\end{array}\right] \succeq 0,
	%\]
	\[
	L_{G}^{(2)}[y] =
	\left[\begin{array}{lll}
		L_{1-x_1x_2}^{(2)}[y]  &\quad  L_{x_1 + x_2}^{(2)}[y]  \\
		L_{x_1 + x_2}^{(2)}[y] &\quad   L_{x_1^2-x_2^2}^{(2)}[y]
	\end{array}\right],
	\]
	where
	\[
	L_{1-x_1x_2}^{(2)}[y] =
	\left[\begin{array}{lll}
		y_{00} - y_{11}  &\quad  y_{10} - y_{21} &\quad y_{01} - y_{12}  \\
		y_{10} - y_{21}  &\quad  y_{20} - y_{31} &\quad y_{11} - y_{22}  \\
		y_{01} - y_{12}  &\quad  y_{11} - y_{22} &\quad y_{02} - y_{13}
	\end{array}\right],
	\]
	\[
	L_{x_1+x_2}^{(2)}[y] =
	\left[\begin{array}{lll}
		y_{10} + y_{01}  &\quad  y_{20} + y_{11}  &\quad  y_{11} + y_{02}  \\
		y_{20} + y_{11}  &\quad  y_{30} + y_{21}  &\quad  y_{21} + y_{12}  \\
		y_{11} + y_{02}  &\quad  y_{21} + y_{12}  &\quad  y_{12} + y_{03}
	\end{array}\right],
	\]
	\[
	L_{x_1^2-x_2^2}^{(2)}[y] =
	\left[\begin{array}{lll}
		y_{20} - y_{02}  &\quad  y_{30} - y_{12}  &\quad  y_{21} - y_{03}  \\
		y_{30} - y_{12}  &\quad  y_{40} - y_{22}  &\quad  y_{31} - y_{13}  \\
		y_{21} - y_{03}  &\quad  y_{31} - y_{13}  &\quad  y_{22} - y_{04}
	\end{array}\right] .
	\]

	% The following is the matrix version of Putinar's Positivstellensatz.
	
	%\begin{thm}
	
	%	Let $\mathcal{G}$ be a set of symmetric  polynomial matrixs such that $\mathrm{QM}[\mathcal{G}]^{\ell}$ is archimedean. If $F \in \mathcal{S} \mathbb{R}[x]^{\ell}$ is positive definite on $S_G$, then $F \in \mathrm{QM}[\mathcal{G}]^{\ell}$.
	%\end{thm}
	
	%\subsection{The Moment-SOS hierarchy for PMO}
	%For $G\in \mathcal{S} \mathbb{R}[x]^{m}$ as in \reff{nsdp},
	
	For a degree $k$, the $k$th order SOS relaxation for solving \reff{nsdp} is
	\be \label{sos:local}
	\left\{\begin{array}{cl}
		\max & \gamma \\
		\text { s.t. } & f-\gamma \in \mathrm{QM}[G]_{2k} .
	\end{array}\right.
	\ee
	%For the polynomial matrix $G(x)$, writing $G(x)=\sum_{\gamma \in \mathbb{N}^n} G_\gamma x^\gamma$, for some finite family of real symmetric matrices $\left\{G_\gamma\right\}_\gamma \subset \mathcal{S}_m$, we also define the ( $m$-block) $s_k$-vector $G y$ by $(G y)_\alpha:=\sum_\gamma G_\gamma y_{\alpha+\gamma}$, for all $\alpha$ with $|\alpha| \leq$ $k$. That is, each entry $(G y)_\alpha,|\alpha| \leq k$, is a $m \times m$ matrix. Then, the localizing matrix $M_k(G \bar{y})$ has the block structure $\left\{(G y)_{\alpha+\beta}\right\}_{\alpha, \beta}$, with $|\alpha|,|\beta| \leq k$; equivalently, from its definition,
	% $L_{G}^{(k)}[y]$ is obtained from the moment matrix $M_k(y)$ by
	% $$
	% \left.M_k(G y)\right]_{\alpha \beta}=L_y\left(\left[\left(b_k(x) b_k(x)^T \otimes G(x)\right]_{\alpha \beta}\right), \quad|\alpha|,|\beta| \leq k\right.
	% $$
	The dual optimization of \reff{sos:local} is the $k$th order moment relaxation
	\be  \label{mom:local}
	\left\{ \baray{cl}
	\min &  \langle f, y \rangle  \\
	\st  &  L_{G}^{(k)}[y] \succeq 0, \, M_k[y] \succeq 0, \\
	& y_0  =  1,  \,  y \in \mathbb{R}^{\mathbb{N}_{2k}^{n} } .
	\earay \right.
	\ee
	For $k=1,\ldots,$ the sequence of relaxations \reff{sos:local}-\reff{mom:local}
	is called the matrix Moment-SOS hierarchy for solving \reff{nsdp}.
	The optimal value of \reff{sos:local} is the SOS bound of order $k$ and the optimal value of \reff{mom:local} is the moment bound of order $k$,  which we denote by $f_{k,sos}$, $ f_{k,mom}$, respectively.  By the weak duality,
	we have $ f_{k,sos} \leq f_{k,mom}$ for all $k$. The Moment-SOS hierarchy of
	\reff{sos:local}--\reff{mom:local} is said to
	have finite convergence if there exists $N>0$ such that
	$f_{k,sos}=f_{\min}$ for all $k \ge N$.

	\subsection{The main results}

	First, we show that under the \rev{Archimedean property} and some optimality conditions
	for \reff{nsdp}, the Moment-SOS hierarchy of
	\reff{sos:local}--\reff{mom:local} has finite convergence.
	
	\begin{thm}  \label{mianthm1:local}
		Suppose the quadratic module $\qmod{G}$  is \rev{Archimedean}.
		\rev{If the nondegeneracy condition \reff{CQ}, strict complementarity condition \reff{SCC} and
			second order sufficient condition \reff{sosc} hold
			at every minimizer of \reff{nsdp}}, then the matrix Moment-SOS hierarchy
		of \reff{sos:local}-\reff{mom:local} has finite convergence,
		i.e., there exists $N>0$ such that
		$f_{k,sos}=f_{k,mom}=f_{\min}$ for all $k \ge N$.
		Furthermore, we have $f-f_{\min}\in \qmod{G}$.
	\end{thm}

	In computational practice, the flat extension or truncation  condition (see \cite{CF05,HenLas05,hdlb,Lau05,nie2013certifying})
	is often used to detect finite convergence and to extract minimizers.
	Suppose $y^{*}$ is a minimizer of $(\ref{mom:local})$ for a relaxation order $k$.
	If there exists an integer $t \in[d_G, k]$ such that
	\be \label{FT:y*}
	\operatorname{rank} M_{t} [ y^* ] \, = \,
	\operatorname{rank} M_{t-d_G}[ y^* ],
	\ee
	then we can get one or several minimizers for \reff{nsdp}
	(see \cite{hdlb}).
	When \reff{FT:y*} holds,
	% we have the decomposition
	%\[
	%y^*|_{2t} = a_1 [\bpm \tau_1 \\ v_1 \epm]_{2t}
	%+ \cdots + a_r [\bpm \tau_r \\ v_r \epm]_{2t}
	%\]
	%for positive scalars $a_i > 0$
	%and distinct points $(\tau_i, v_i) \in \widetilde{K}$.
	the truncation $y^*|_{2t}$ is a moment vector
	for a finitely $r$-atomic probability measure $\mu^*$ whose support is contained in $S(G)$.
	That is, there exist points $u_1,\ldots, u_r \in S(G)$ such that
	\[
	\mu^*=\sum_{j=1}^r \gamma_j \delta_{u_j}, \quad
	\sum_{j=1}^r \gamma_j=1 \quad \gamma_j>0, \quad j=1, \ldots, r,
	\]
	where $\delta_{u_j}$ denotes the unit Dirac measure supported at $u_j$.
	One can further show that $f_{k,mom}=f_{\min}$ and $u_1, \dots, u_r$
	are minimizers of \reff{nsdp}.
	We refer to \cite{hdlb} and \cite[Chapter~10]{niebook} for details.

	Our second main result is to prove that under the \rev{Archimedean property} and optimality conditions,
	every minimizer of the moment relaxation \reff{mom:local}
	must have a flat truncation,  when the relaxation order $k$ is big enough.
	Recall that an ideal $I\subseteq \mR[x]$ is said to be zero dimensional
	if its complex variety $V_{\mC}(I)$
	is a finite set.

	\begin{thm}  \label{mianthm2:local}
		Suppose the quadratic module $\qmod{G}$  is \rev{Archimedean}.
		\rev{If the nondegeneracy condition \reff{CQ}, strict complementarity condition \reff{SCC} and
			second order sufficient condition \reff{sosc} hold
			at every minimizer of \reff{nsdp}},
		then we have:
		\bit
		\item[(i)] The support   $\supp{\qmod{G}+\ideal{f-f_{\min}}}$ is zero dimensional.
		%the ideal  $(\qmod{G}+\ideal{f-f_{\min}}) \cap-(\qmod{G}+\ideal{f-f_{\min}})$ has dimension $0$.	
		
		\item[(ii)] Every minimizer of the moment relaxation \reff{mom:local}
		must have a flat truncation, when $k$ is sufficiently large.
		
		\eit	
	\end{thm}

	\subsection{Difficulties from the scalar case to the matrix case}
	\label{proofdif}

	For the scalar constrained polynomial optimization problem
	\be  \label{sp0}
	\left\{ \baray{rl}
	\min & f(x)  \\
	\st &  	g_1(x)\geq 0,\cdots, g_m(x)  \geq 0, \\
	\earay \right.
	\ee
	%where $f,g_1,\dots,g_m\in\mR[x]$,
	it was shown in \cite{nieopcd} that
	the Moment-SOS hierarchy for solving \reff{sp0}  has finite convergence
	if the linear independence constraint qualification,
	strict complementarity and second order sufficient conditions in nonlinear programming
	hold at every minimizer of \reff{sp0}.
	We refer to \cite{nieopcd} for the details of these conditions.
	The proof uses the boundary Hessian condition (BHC)
	by Marshall \cite{mar06}, which requires a local smooth parametrization at each minimizer.
	When the BHC holds at every minimizer, Marshall \cite{mar06} showed that
	the assumptions of Theorem~\ref{scal:bhc} by Scheiderer are satisfied
	and thus $f-f_{\min}\in \qmod{g_1,\dots, g_m}$
	by Theorem \ref{scal:bhc}\footnote{
		In fact, Marshall has proved a weaker condition rather than item (ii)  in Theorem \ref{scal:bhc},
		which is also enough to get the result.
		Item (ii) was shown by Scheiderer in \cite{Sch09}.}.
	The proof of the finite convergence for the scalar optimization \reff{sp0}
	in \cite{nieopcd} is to show that if the above mentioned
	optimality conditions hold at every minimizer,
	then the BHC also holds, which then implies the finite convergence.
	
	\rev{
		%Let $S(g)$ be the feasible set of \reff{sp0} and let $u$ be feasible for \reff{sp0}.
		%		The BHC is said to hold at $u$ if:
		%		 (i) there exist a neighborhood $U$ of $u$ and $t_1,\dots,t_n$ with  $t_i=g_{v_i}$ $(i=1,\dots,r)$, $r\leq n$ such that $S(g)\cap U$ is defined by $t_1\geq 0, \dots,t_r\geq 0$;		
		%		(ii) Expand $f$ locally
		%		around $u$ as $f=f_0+f_1+f_2+f_3+\cdots$, with each $f_i$ being homogeneous of degree $i$ in $t$. Then, the linear form $f_1=c_1t_1+\cdots+c_rt_r$ for some positive constants $c_1>0,\dots,c_r>0$,  and the quadratic form $f_2(0,\ldots,0, t_{r+1},\ldots, t_n)$ is positive definite in $t_{r+1},\dots, t_n$.
	}
	%\textcolor{blue}{(Should we remove this part? It is essentially not used later at all.)}

	Theorem \ref{mianthm1} (also Theorem \ref{mianthm1:local}) can be viewed as
	a matrix analogue of finite convergence theory for the \rev{scalar} case in \cite{nieopcd}.
	However, Theorem \ref{mianthm1} cannot be \rev{exposed} by a straightforward extension of the
	results for the scalar case. The major differences and difficulties
	are shown  below.

	\bit

	\item[(i)]
	The optimality conditions in nonlinear semidefinite optimization have
	geometric properties that are different from those for scalar constrained optimization.
	When  $G = \diag(g_1, \ldots,g_m)$ is diagonal and the NDC holds at a feasible point $u\in \mR^n$,
	then there exists at most one index $i \in [m]$ such that $g_i(u)=0$,
	regardless of whether the gradients $\nabla g_1(u),\dots, \nabla g_m(u)$
	are linearly independent or not, as we have seen in Section \ref{seccq}.
	%Thus, the finite convergence result for the scalar case shown in \cite{nieopcd}
	%cannot be applied to prove Theorem~\ref{mianthm1}.
	Moreover, the second order sufficient condition \reff{sosc} for \reff{nsdp}
	involves the above mentioned ``$H$-term",
	which further complicates the proof. In contrast,
	this term does not appear in the  second order sufficient condition
	for nonlinear scalar constrained optimization.

	%The optimality conditions in nonlinear semidefinite optimization differ significantly from those in nonlinear programming, involving more complex geometry. For instance, the definition of the NDC is quite technical, particularly from an algebraic perspective. Additionally, the second-order sufficient condition \reff{sosc} in nonlinear semidefinite optimization includes the well-known and often-mentioned "$H$-term", which further complicates the proof.
	%	\textcolor{blue}{(Should we remove these parts? This does not explain anything, and is also repeated.)}

	\item[(ii)] The feasible set $S(G)$ of \reff{nsdp}
	usually has singular points on its boundary, even if the NDC holds.
	This is the case even if $G$ is linear.
	Since nonsingularity is a basic assumption in
	the finite convergence theory for scalar constrained optimization in \cite{nieopcd},
	it cannot be applied for  polynomial matrix optimization.
	We refer to \rev{relation} \reff{intexm} for such an example.

	\item[(iii)]
	To prove Theorem~\ref{mianthm1}, one may naturally look for appropriate scalarization
	of the polynomial matrix inequality $G(x)\succeq 0$ and then apply
	the finite convergence theory for the scalar case in \cite{nieopcd}.
	For this purpose,  one need to find polynomials $g_1,\dots,g_s\in \mR[x]$ such that
	\[
	S(G)=\{x\in \mR^n\mid  g_1(x)\geq 0,\dots,g_s(x)\geq 0\},
	\]
	and then \reff{nsdp} is equivalent to \reff{sp0}.
	%\be  \label{nsdp:equi}
	%\left\{ \baray{rl}
	%\min & f(x)  \\
	%\st &  g_i(x)\geq 0~(i=1,\dots,s), \\
	%\earay \right.
	%\ee
	To apply the results for the scalar case, we need each $g_i(x)\in \qmod{G}$.
	We also need to prove that if the NDC, SCC
	and SOSC hold at a minimizer $u$ of \reff{nsdp},
	then the linear independence constraint qualification condition, strict complementarity condition
	and second order sufficient condition in nonlinear programming hold at $u$ for \reff{sp0}.
	However, the typical scalarization may either violate the optimality conditions
	%(i.e., produces singular loci),
	or it may need polynomials that do not belong to $\qmod{G}$.

	For instance,  consider the polynomial matrix
	\[
	G(x)=\left[\begin{array}{lll}
		1- \|x\|^2 & ~0  & ~ 0\\
		0&~ x_1  & ~x_2\\
		0&~x_2   &~x_3 \\
	\end{array}\right].
	\]
	Then, $G(x)\succeq 0$ if and only if
	\[
	1-x_1^2-x_2^2-x_3^2\geq 0, ~\rev{\hat{G}(x)}=\left[\begin{array}{ll}
		x_1  & x_2\\
		x_2   &x_3 \\
	\end{array}\right]\succeq 0.
	\]
	A natural scalarization for $\rev{\hat{G}(x)}\succeq 0$ is that all its principal minors are nonnegative. Hence, the matrix inequality $G(x)\succeq  0$
	can be equivalently given by scalar inequalities
	\[
	1- \|x\|^2  \geq 0,~x_1\geq 0,~x_3\geq  0, ~x_1x_3-x_2^2\geq 0.
	\]
	
	%A natural scalarization is to consider the characteristic polynomial
	%%$\det[ \lambda I_m-G^*(x)]$  of $G^*(x)$, which can be expressed as
	%\[
	%\det[\lambda  I_2-G^*(x)] \,= \, \lambda^2-(x_1+x_3)\lambda+x_1x_3-x_2^2.
	%\]
	%Since $G^*(x)$ is symmetric, all its eigenvalues are real. So, $G^*(x)\succeq 0 $ if and only if $x_1+x_3\geq  0$, $x_1x_3-x_2^2\geq 0$. Hence, the matrix inequality $G(x)\succeq  0$
	%can be equivalently given by scalar inequalities
	%\[
	%1-x_1^2-x_2^2-x_3^2\geq 0,~x_1+x_3\geq  0, ~x_1x_3-x_2^2\geq 0.
	%\]
	However, \rev{the polynomial $x_1x_3-x_2^2$ does not belong to $\qmod{G}$,} as shown in Example \ref{exmdet}.
	
	Another scalarization can be motivated by \cite[Proposition 9]{sk09}.
	It gives finitely many scalar polynomials in $\qmod{G}$ to describe the set $S(G)$.
	By following the procedure in \cite{sk09}, the set $\qmod{G}$ can be equivalently given by
	\[
	x_1\geq 0,~x_3\geq 0,~x_1(x_1x_3-x_2^2)\geq 0,
	\]
	\[
	x_1+2x_2+x_3\geq 0,~x_3(x_1x_3-x_2^2)\geq 0,
	\]
	\[
	(x_1+2x_2+x_3)(x_3(x_1+2x_2+x_3)-(x_2+x_3)^2)\geq 0.
	\]
	However, the LICQC fails at $(0,0,0)$.
	%since the gradient of
	%$x_1(x_1x_3-x_2^2)$ at $u$ is zero.
	%	

	\item[(iv)] The proof of finite convergence theory for the scalar case
	relies on the boundary Hessian condition by Marshall \cite{mar06}.
	% which implies conditions appearing in Theorem \ref{scal:bhc}.
	However, there does not exist a matrix version of the BHC.
	% and the corresponding
	%local-global principle for psd matrix polynomials.
	%Hence, Theorem \ref{mianthm1} can also be viewed as a matrix version of local-global principle.
	%Furthermore, Lemma 4.3 is also new in  real algebraic geometry,
	%which establish a SOS representation for a new type of power series.
	%	
	A major difficulty for this is that
	the feasible set $S(G)$ is typically singular at a minimizer $u$,
	even if $G$ is linear.
	The reason is that at the minimizer $u$, the rank of $G(u)$
	is usually less than $m-1$ (see \cite{NRS10}).
	The BHC typically requires that the feasible set is nonsingular around $u$,
	so that a smooth  parametrization is available.

	\item[(v)] Recall that a quadratic module $M$ is said to be finitely generated
	if there exists a polynomial tuple $g=(g_1,\dots,g_m)$ such that $M=\qmod{g}$.
	However, the quadratic module $\qmod{G}$ is typically not finitely generated,
	as shown in \cite{cj}. The earlier results about SOS representations for
	nonnegative polynomials (e.g., Theorem~\ref{scal:bhc})
	is usually applicable to finitely generated quadratic modules.
	This issue gives another difficulty for the proof of Theorem~\ref{mianthm1}.

	\eit

	\subsection{Some examples}
	
	We give an illustrating example to demonstrate the
	finite convergence of the matrix Moment-SOS hierarchy \reff{sos:local}-\reff{mom:local}.

	\begin{exm}
		Consider the matrix constrained optimization problem:
		\be  \label{sucex2}
		\left\{ \baray{rl}
		\min & x_1x_3-x_2^2+x_1+x_3 \\
		\st &  \left[\begin{array}{lll}
			1- \|x\|^2  & 0  &  0\\
			0  &   x_1  & x_2\\
			0  &  x_2   & x_3 \\
		\end{array}\right]\succeq 0,
		\earay \right.
		\ee
		The optimal value $f_{\min} = 0$ and the unique minimizer is the origin $0$.
		One can verify the NDC, SCC and SOSC all hold at $0$.
		The Moment-SOS hierarchy of \reff{sos:local}-\reff{mom:local}
		has finite convergence. A numerical experiment
		%by {\tt YALMIP}
		indicates that $f_{2,sos} = f_{2,mom}=8.7136\cdot10^{-14}$
		and we get the minimizer $(0.0000,0.0000,0.0000)$.
	\end{exm}	
	
	%%%%%%%%%%%%%%%%%%%%
	\iffalse	
	The objective is the Robinson form which is nonnegative but not SOS (cf. [23]). The minimum $f_{\min }=0$, and the global minimizers are
	$$
	\frac{1}{\sqrt{3}}( \pm 1, \pm 1, \pm 1), \frac{1}{\sqrt{2}}( \pm 1, \pm 1,0), \frac{1}{\sqrt{2}}( \pm 1,0, \pm 1), \frac{1}{\sqrt{2}}(0, \pm 1, \pm 1) .
	$$
	
	The unit sphere is smooth, so the constraint qualification condition holds at every feasible point. There is no inequality constraint, so strict complementarity is automatically satisfied. It can be verified that the second order sufficiency condition (1.7) holds on all the global minimizers. For instance, at $u=\frac{1}{\sqrt{3}}(1,1,1)$,
	$$
	\nabla_x^2 L(u)=\frac{4}{9}\left(3 \cdot\left[\begin{array}{lll}
		1 & 0 & 0 \\
		0 & 1 & 0 \\
		0 & 0 & 1
	\end{array}\right]-\left[\begin{array}{l}
		1 \\
		1 \\
		1
	\end{array}\right]\left[\begin{array}{l}
		1 \\
		1 \\
		1
	\end{array}\right]^T\right), \quad G(u)^{\perp}=\left[\begin{array}{l}
		1 \\
		1 \\
		1
	\end{array}\right]^{\perp} .
	$$
	\fi
	%%%%%%%%%%%%%%%%%%	

	%\begin{exm}
	% Consider the optimization problem:
	% $$
	% \left\{\begin{array}{l}
		% \min x_1^6+x_2^6+x_3^6+3 x_1^2 x_2^2 x_3^2-x_1^4\left(x_2^2+x_3^2\right)-x_2^4\left(x_3^2+x_1^2\right)-x_3^4\left(x_1^2+x_2^2\right) \\
		% \text { s.t. } x_1^2+x_2^2+x_3^2=1 .
		% \end{array}\right.
	% $$
	%\end{exm}

	We would like to remark that none of the optimality conditions
	can be dropped in Theorems \ref{mianthm1:local} and \ref{mianthm2:local}.
	The following are the counterexamples \rev{ that illustrate this}.
	
	\begin{exm}\label{exmdet}	
		(i) Consider the optimization problem:
		\be \nn %\label{ex2}
		\left\{\begin{array}{l}
			\min~ 3 x_1+2 x_2 \\
			\st \left[\begin{array}{ll}
				x_1  &~ x_3  \\
				x_3&~ x_1^2-x_2^2- \|x\|^4   \\
			\end{array}\right]\succeq 0.
		\end{array}\right.
		\ee
		The minimum value $f_{\min} = 0$ and the origin $0$ is the unique minimizer.
		Note that $G(0)=0$. By direct calculation, one can see
		\[
		\operatorname{Im} \nabla G(0)=\left\{\bbm d_1 & d_2\\d_2 &0\\ \ebm :\, d_1,d_2\in \mR\right\},\,
		\operatorname{lin}(T_{\mathcal{S}_{+}^2}(G(0))) = \bbm 0 & 0 \\ 0 & 0 \ebm,
		\]
		\[
		\operatorname{Im} \nabla G(0)+\operatorname{lin}(T_{\mathcal{S}_{+}^2}(G(0)))  =
		\operatorname{Im} \nabla G(0)\ne \mathcal{S}^2.
		\]
		Thus, the NDC fails at $0$.
		The feasible set has nonempty interior, e.g.,
		the matrix $G(x)$ is positive definite for $x=(\frac{1}{2},0,0)$.
		Hence, the optimal value of the SOS relaxation \reff{sos:local} is achievable.
		We show that the hierarchy of \reff{sos:local}-\reff{mom:local}
		has no finite convergence. Suppose otherwise it did have,
		then there exist $\sigma\in \Sigma[x]$ and polynomial vectors
		$v_{i} \in \re[x]^2$, $i=1,\dots,N$,  such that
		\be \nonumber
		3 x_1+2 x_2 \, = \, \sigma+\sum\limits_{i=1}^{N} v_i^TGv_i.
		\ee
		Substituting $x_3=0$ in the above, we get
		\[
		3 x_1+2 x_2\in \qmod{x_1,x_1^2-x_2^2-\left(x_1^2+x_2^2\right)^2}
		\]
		which is not true as shown in \cite[Example 3.3]{nieopcd}.
		Therefore, the hierarchy of \reff{sos:local}-\reff{mom:local}
		fails to have finite convergence. \\
		%%%%%%%%%%%%%%%%%%%%%%%%%%%%%%%%%%%%%%%%%%%%%%%%%%%%%%%%
		\noindent
		(ii) Consider the optimization problem:
		\be \nn
		\left\{ \baray{rl}
		\min & x_1x_3-x_2^2 \\
		\st &
		\left[\begin{array}{lll}
			1- \|x\|^2 & 0  &  0\\
			0& x_1  & x_2\\
			0&x_2   &x_3 \\
		\end{array}\right]\succeq 0.
		\earay \right.
		\ee
		The minimum value is $0$.
		At the minimizer $0$,
		%we have
		%\[
		%G(0)=\left[\begin{array}{lll}
			%1 & 0  &  0\\
			%0& 0  & 0\\
			%0&0   &0 \\
			%\end{array}\right].
			%\]
			%a basis of $\operatorname{ker} G(0)$ is
			%given by the columns of $E = \bbm e_2 & e_3 \ebm$.
			%Since the gradient vectors of $x_1,x_2,x_3$ are linearly independent at $u$,
			%we know that the nondegeneracy condition holds at $u$.
			one can verify that  the nondegeneracy condition holds.
			The Lagrange multiplier matrix at $0$ is $\Lambda=0$,
			so the SCC fails. We show the hierarchy of \reff{sos:local}-\reff{mom:local}
			does not have finite convergence.
			For $x=(1/2,0,1/2)$, the matrix $G(x)$ is positive definite,
			so there are interior points. Hence,
			the optimal value of the SOS relaxation \reff{sos:local} is achievable.
			Suppose otherwise the hierarchy of \reff{sos:local}-\reff{mom:local}
			did have finite convergence, then we have
			% \be \nonumber
			% 3 x_1+2 x_2=\sum\limits_{i=1}^{r_1}\sigma_i^2+\sum\limits_{i=1}^{r_2}(v_{i1}^2x_1+v_{i2}^2(x_1^2-x_2^2-\left(x_1^2+x_2^2+x_3^2\right)^2)+2v_{i1}v_{i2}x_3).
			% \ee
			$
			x_1x_3-x_2^2\in \qmod{G},
			$
			and hence $x_1x_3\in \qmod{1-x_1^2-x_3^2,x_1,x_3}$, i.e.,
			there exist $\sigma_0,\sigma_1,\sigma_2,\sigma_3\in \Sigma[x_1,x_3]$ such that
			\[
			x_1x_3=\sigma_0+\sigma_1(1-x_1^2-x_3^2)+\sigma_2 x_1+\sigma_3x_3.
			\]
			%\ter{By substituting $(x_1,x_3)=0$ into the above identity,
				%we have that $ \sigma_0(0)=\sigma_1(0)=0$.
				%Since the left hand has no linear monomials $x_1$,
				%$x_3$,  $ \sigma_2(0)=\sigma_3(0)=0$.
				%Hence, we know that $\sigma_1\cdot(-x_1^2-x_3^2)$,
				%$\deg(\sigma_2 x_1)\geq 3$, $\deg(\sigma_3x_3)\geq 3$. Compare the lowest degree terms on both sides, we know that $x_1x_3$ is equal to the lowest degree terms of $\sigma_0+\sigma_1$, which  is SOS. Hence, this cannot hold.}
			The above representation implies that
			$\sig_0, \sig_1, \sig_2$ must vanish on $\{x_3=0, 1\ge x_1  \ge 0\}$, and
			$\sig_0, \sig_1, \sig_3$ must vanish on $\{x_1=0, 1\ge x_3  \ge 0\}$.
			This further implies that
			$\sig_0, \sig_1, \sig_2$ must vanish on $\{x_3=0 \}$ and
			$\sig_0, \sig_1, \sig_3$ must vanish on $\{x_1=0 \}$.
			Since they are all SOS, we know that
			$\sig_0, \sig_1$ are multiples of $x_1^2x_3^2$,
			$\sig_2$ is a multiple of $x_3^2$,
			and $\sig_3$ is a multiple of $x_1^2$.
			Hence, there are $\sigma_0^\prime,\sigma_1^\prime,
			\sigma_2^\prime, \sigma_3^\prime \in \Sigma[x_1,x_3]$ such that
			\[
			1= x_1x_3 \sigma_0^\prime+ x_1x_3 \sigma_1^\prime(1-x_1^2-x_3^2)+
			\sigma_2^\prime x_3+\sigma_3^\prime x_1.
			\]
			However, this cannot hold since the right hand side
			vanishes on $x_1=x_3 = 0$. \\
			%%%%%%%%%%%%%%%%%%%%%%%%%%%%%%%%%%%%%%%%%%%%%%%%%%%%%%%%%%%%%%%%%%%%%%%%%%%%%%
			\noindent
			(iii) Consider the optimization problem:
			\be \nn % \label{ex3}
			\left\{\begin{array}{rl}
				\min & x_1^4 x_2^2+x_1^2 x_2^4+x_3^6+x_4^6-3 x_1^2 x_2^2 x_3^2+
				0.01\cdot\left(x_1^6+x_2^6+x_3^6\right) \\
				\st  & \left[\begin{array}{rr}
					1-x_4^2  &~  x_3x_4  \\
					x_3x_4   & ~ 1- \|x\|^2  \\
				\end{array}\right]\succeq 0.
			\end{array}\right.
			\ee
			The minimum value $f_{\min} = 0$ and $0$ is the unique minimizer.
			One can verify that the NDC and SCC hold at $0$, but SOSC fails.
			The optimal value of the SOS relaxation \reff{sos:local} is achievable,
			since there are interior points. The hierarchy of \reff{sos:local}-\reff{mom:local}
			does not have finite convergence. Suppose otherwise it did have,
			then there exist $\sigma\in \Sigma[x]$ and polynomial vectors
			$v_{i} \in \re[x]^2$, $i=1,\ldots,N$, such that
			\be \nonumber
			f=\sigma+\sum\limits_{i=1}^{N}v_i^TGv_i.
			\ee
			Substituting $x_4=0$  in the above, we then get
			\be \nonumber
			x_1^4 x_2^2+x_1^2 x_2^4+x_3^6-3 x_1^2 x_2^2 x_3^2+0.01\cdot( \sum_{i=1}^3 x_i^6)
			\in \qmod{1- \|x\|^2},
			\ee
			which does not hold (see \cite[Example 6.25]{Lau09}).
			Therefore, the hierarchy of \reff{sos:local}-\reff{mom:local}
			fails to have finite convergence.
		\end{exm}

		\subsection{The SOS-convex case}
		
		When $f(x)$ and $-G(x)$ are SOS-convex,
		we show that the lowest order Moment-SOS relaxation is exact.
		Recall that $f(x)$ is SOS-convex if the Hessian $\nabla^2 f(x)$ \rev{belongs to} $\mathcal{S}^{n}\Sig[x]$,
		i.e., there exists a  polynomial matrix $P \in \mathbb{R}[x]^{\ell \times n}$,
		for some $\ell \in \mathbb{N}$, such that $\nabla^2 f =  P^T P$.
		The polynomial matrix $-G(x)$ is said to be SOS-convex
		if for every $\xi \in \mathbb{R}^{m}$, there exists a  polynomial matrix $F(x) \in \mathbb{R}[x]^{t \times n}$, for some $t \in \mathbb{N}$, such that
		$$
		-\nabla^2\left(\xi^T G(x) \xi\right)=F(x)^T F(x) .
		$$
		We remark that the coefficients of the above $F(x)$ may depend on $\xi$.
		We refer to \cite{HilNie08,kism,niepolymat,niebook}
		for related work on SOS-convex  polynomial matrices.

		\begin{thm}
			Suppose $f\in \mR[x]_{2d}$ and $G \in \mathcal{S}^m\mathbb{R}[x]$ with $\deg(G)\leq 2d$. 	
			Assume $f$ and $-G$ are SOS-convex. Then, we have:
			\bit
			
			\item [(i)]
			The lowest order moment bound in \reff{mom:local} is exact,
			i.e., $f_{d,mom} = f_{\min}$. Moreover, if there exists $u\in \mR^n$
			such that $G(u)\succ 0$, then the lowest order SOS bound in \reff{sos:local} is also exact,
			i.e., $f_{d,sos} = f_{\min}$.
			
			\item [(ii)]
			If $y^*$ is a minimizer of \reff{mom:local}, then the point
			$u^*=(y_{e_1}^*,\dots,y_{e_n}^*)$ is a minimizer of \reff{nsdp}.

			\eit
		\end{thm}
		\begin{proof}
			(i) Consider the projection of the $d$th order moment relaxation \reff{mom:local}
			\[
			\Gamma:=\left\{\begin{array}{l|l}
				x \in \mathbb{R}^n & \begin{array}{c}
					\exists~ y \in \mathbb{R}^{\mathbb{N}_{2 d}^n},~~y_0=1, \\
					L_{G}^{(d)}[y] \succeq 0, M_d[y] \succeq 0, \\
					x=\left(y_{e_1}, \ldots, y_{e_n}\right)
				\end{array}
			\end{array}\right\}.
			\]
			Under the assumptions that $-G$ is SOS-convex,
			it was shown in \cite{niepolymat} that $S(G)=\Gamma$.
			For each $\epsilon>0$,  there exists a feasible $y$ of \reff{mom:local} such that
			$
			\langle f, y\rangle < f_{d,mom}+\epsilon .
			$
			Note that $y_0=1$ and $M_d[y] \succeq 0$.
			Since $f$ is SOS-convex, Jensen's inequality holds
			(see \cite{Las09} or \cite[Chapter~7]{niebook})
			for the point $u:=\left(y_{e_1}, \ldots, y_{e_n}\right)$, i.e.,
			$$
			f(u) \leq\langle f, y\rangle<f_{d,mom}+\epsilon.
			$$
			Since $u \in S(G)$, we get $f_{\min } \leq f(u) \leq f_{d,mom}+\epsilon$.
			This holds for all $\epsilon>0$, so $f_{\min } \leq f_{d,mom}$.
			On the other hand, $f_{d,mom} \leq f_{\min}$,
			therefore we must have $f_{d,mom}=f_{\min}$.
			Moreover, if there exists $u\in \mR^n$ such that $G(u)\succ 0$,
			then the Slater condition holds for the primal-dual pair \reff{sos:local}--\reff{mom:local}.
			Then, the strong duality holds between them, so $f_{d,sos}=f_{d,mom}=f_{\min}$.
			
			(ii) Suppose $y^*$ is a minimizer of \reff{mom:local}.
			Since $f$ and $-G$ are SOS-convex, Jensen's inequality
			(see \cite{Las09} or \cite[Chapter~7]{niebook})
			implies that
			\[
			f_{\min } \leq f(u^*) \leq\left\langle f, y^*\right\rangle=f_{d,mom}=f_{\min}.
			\]
			Since $u^*\in S(G)$, we know that $u^*$ is a minimizer of \reff{nsdp}.

		\end{proof}

		\begin{exm}
			Consider the optimization problem:
			\be  \nn %%  \label{sucex3}
			\left\{ \baray{rl}
			\min & \frac{1}{3}(x_1^4+x_2^4+x_3^4)+x_1^2x_2^2+(x_1-1)^2+(x_2-1)^2+(x_3-1)^2 \\
			~\\
			\st &  \left[\begin{array}{rrr}
				2-x_1^2-2 x_3^2 & 1+x_1 x_2 & x_1 x_3 \\
				1+x_1 x_2 & 2-x_2^2-2 x_1^2 & 1+x_2 x_3 \\
				x_1 x_3 & 1+x_2 x_3 & 2-x_3^2-2 x_2^2
			\end{array}\right] \succeq 0.
			\earay \right.
			\ee
			Both $f$ and $-G$ are SOS-convex (see \cite{niepolymat}).
			The relaxation pair \reff{sos:local}-\reff{mom:local} has finite convergence for $k=2$.
			A numerical experiment
			indicates that $f_{2,sos} = f_{2,mom} \approx 1.1321$
			and we can get the minimizer $(0.3977,0.3096,0.5057)$.
		\end{exm}

		%\section{Local SOS representations}
		%\label{sec:localsos}

		\section{Proofs of the main theorems}
		\label{sec:proof}
		
		%This section gives proofs for Theorems~\ref{mianthm1} and \ref{mianthm2},
		%with the usage of results shown in the previous sections.
		This section gives proofs for Theorems~\ref{mianthm1} and \ref{mianthm2}.
		The proof of Theorem \ref{mianthm1} uses SOS type representations
		for nonnegative polynomials,
		%% having only finitely many zeros,
		which is given in Theorem~\ref{scal:bhc} by Scheiderer~\cite{Sch05}.
		The proof of Theorem \ref{mianthm2} uses some results
		in the proof of Theorem~\ref{mianthm1} and exploits some properties of moment matrices. The proof of Theorem~\ref{mianthm1} is given in Section \ref{proofthm1}
		and the proof of Theorem \ref{mianthm2} is given in Section \ref{proofthm2}.
		Sections~\ref{sec:eqpmo} and \ref{sec:localsos} give some preliminary results for the proofs.

		\subsection{Overview of the proofs }
		\label{sec:proofs}
		
		Recall that the feasible set \rev{of \reff{nsdp} is}  $S(G)=\{x\in \mR^n:  G(x)\succeq 0\}$.
		For a quadratic module $M$, the notation $\widehat{M}_u$ denotes the quadratic module
		generated by the image of $M$ in the completion ring of the localization of $\mR[x]$ at $u$ (see Section \ref{sec:qm} for the  notation).
		The main idea in the proof of Theorem \ref{mianthm1} is to find a polynomial tuple $\Phi= \{\phi_1,\dots,\phi_s\}$ (see \reff{fin:phi}) such that:

		\bit
		
		\item[(i)] $\qmod{\Phi}$ is \rev{Archimedean} and each $\phi_i\in \qmod{G}$.

		\item[(ii)]	$S(G)=\{x\in \mR^n: \, \phi_i(x)\geq 0,~ i=1,\dots,s\}$.
		
		\item[(iii)]
		For every minimizer $u$ of \reff{nsdp},	we have
		$f-f_{\min} \in \widehat{\qmod{\Phi}}_u$ (also known as the local SOS representation).

		\item[(iv)] %For every minimizer $u$ of \reff{nsdp},
		There exist a neighborhood $U$ of $u$ in $\mR^n$ and an element $a \in \qmod{\Phi}$
		such that
		\[
		\{x\in \mR^n: a(x) \geq 0\} \cap V_{\mR}(f-f_{\min}) \cap U \subseteq S(G).
		\]

		\eit
		%
		%under the conditions that the NDC, SCC, SOSC hold
		%at every minimizer of \reff{nsdp}.
		%
		\rev{The above claims (i)--(iv)} show that the assumptions in Theorem \ref{scal:bhc} are satisfied
		for the quadratic module $\qmod{\Phi}$ and the set $S(G)$.
		Then, Theorem \ref{scal:bhc} implies that $f-f_{\min}\in \qmod{\Phi}$.
		Since each $\phi_i\in \qmod{G}$, we have $f-f_{\min}\in \qmod{G}$,
		which completes the proof of Theorem \ref{mianthm1}.
		It is worthy to remark that we do not prove $f-f_{\min}\in\qmod{G}$ directly,
		as \rev{the latter quadratic module} is not finitely generated generally \cite{cj}.
		To obtain the above polynomial set $\Phi= \{\phi_1,\dots,\phi_s\}$,
		we need to thoroughly investigate the relationship between the optimality conditions
		(i.e., NDC, SCC, SOSC) and the local SOS representations at each minimizer.
		The proof can be done in four major steps:

		\bigskip
		\noindent
		{\bf Step 1.}  For optimization \reff{nsdp}, the nondegeneracy condition \reff{CQ}
		is quite tricky. The second-order sufficient condition \reff{sosc}
		includes a challenging ``$H$-term", which makes it difficult to obtain local SOS representations.
		In Section \ref{sec:eqpmo}, we give a  locally  equivalent reformulation for \reff{nsdp}.
		Suppose $u$ is a local minimizer of \reff{nsdp},  with $\rank\, G(u) =r$.
		We prove that there exists a  polynomial matrix $T\in \mathcal{S}^{m-r}\qmod{G}$  such that
		in a neighborhood $U$ of $u$, \reff{nsdp} is equivalent to
		\be   \label{intro1}
		\left\{ \baray{rl}
		\min & f(x)  \\
		\st &  	T(x)\succeq 0,~x\in U. \\
		\earay \right.
		\ee
		Furthermore, we show the optimality conditions for \reff{nsdp} and \reff{intro1}
		are equivalent (see Proposition  \ref{propt}).
		By doing this, the NDC holds at $u$ if and only if
		$\operatorname{ker} \nabla T(u)^*=\{0\}$
		and the challenging ``$H$-term" in  \reff{sosc} disappears.
		We remark that the property $T\in \mathcal{S}^{m-r}\qmod{G}$
		is important since it preserves the algebraic structure.
		
		\bigskip
		\noindent
		{\bf Step 2.}
		In Section \ref{sec:localsos}, we prove that if the NDC, SCC and SOSC
		hold at a local minimizer $u$, then $f-f_{\min} \in \widehat{\qmod{G}}_u$,
		by using the reformulation \reff{intro1}. That is, there exist
		$g_u^{(1)},g_u^{(2)},\dots,g_u^{(t_u)} \in \qmod{G}$
		and $\sigma_u^{(1)},\dots,\sigma_u^{(t_u)}$
		which are SOS in  $\widehat{\mathbb{R}[x]}_{u}$, such that
		\[
		f-f_{\min}=\sigma_u^{(1)}g_u^{(1)}+\cdots+\sigma_u^{(t_u)}g_u^{(t_u)}.
		\]
		This is shown in  Theorem~\ref{mat:bhc}. This proof is completed in two steps:
		
		\medskip
		\noindent
		{\bf Step 2a.} We first prove that if a power series $f$ in the new variables
		$v=(v_1,\dots,v_n)$ can be expanded as in \reff{expand0},
		then $f$ lies in the quadratic module in $\mathbb{R}[[v]]$,
		which is generated by the image of $\qmod{V}$, i.e., $f\in \widehat{\qmod{V}}_0$.	
		See Lemma \ref{aule}.
		
		%	Let  $0\leq r \leq n$ be an integer with $\binom{m-r}{2}\leq n$ and let the vector $v=\left(v_1, \ldots, v_n\right)$ consisting of $n$ new variables. Let
		%	$V \in \mathcal{S}^{m-r} \mathbb{R}[v]$  whose upper triangular part is
		%\be \nonumber
		%\left[\begin{array}{clcl}
			%	v_1 &v_{2}&\dots &v_{ m-r } \\
			%	&v_{m-r+1}&\dots & v_{2(m-r)-1} \\
			%	& &\ddots & \vdots \\
			%	& & &v_{\sigma(m-r) } \\
			%\end{array}\right].
			%\ee We first prove that if the power series $f$ can be expanded as
			%\be \label{local3}
			%f=\langle C,  V\rangle+f_2(v)+\sum\limits_{i=3}^{\infty}f_i  \in
			%\mathbb{R}\left[\left[v_1, \ldots, v_n \right]\right],
			%\ee
			%where  $C \in \mathcal{S}_{+}^{m-r}$, the quadratic form $		f_2(0,\ldots,0, v_{\binom{m-r}{2} +1},\ldots, v_n)$
			%	is positive definite in $v_{\binom{m-r}{2} +1},\dots, v_n$, and $f_i$ is the degree-$i$ homogeneous part of  $f$.
			%Then, $f$ lies in the quadratic module %in
			%$\mathbb{R}[[x]]$, which is
			%generated by the image of $\qmod{V}$, i.e., $f\in \widehat{\qmod{V}}_0$.	 Please see Lemma \ref{aule}.
			
			\medskip
			\noindent
			{\bf Step 2b.} Up to an invertible linear coordinate transformation
			and shifting by a constant, we can generally assume that $f_{\min}=0$ and $u=0$.
			Using the reformulation \reff{intro1}, we show that if the NDC, SCC and SOSC hold at $u$,
			then there exist new variables $v_1,\dots,v_n$ with $V\in \mathcal{S}^{m-r}\qmod{G}$
			such that $f$ admits an expansion as in \reff{expand0}. Then, we show that
			$f-f_{\min} \in \widehat{\qmod{T}}_u$ and $f-f_{\min} \in \widehat{\qmod{G}}_u$.
			%% since $V\in \mathcal{S}^{m-r}\qmod{G}$,
			This is done in the proof of Theorem \ref{mat:bhc}.
			We refer to Section~\ref{sec:localsos} for the details.

			\bigskip
			\noindent
			{\bf Step 3.} In Theorem \ref{exist:a}, we show that if the NDC, SCC and SOSC
			hold at $u$, then   there are a neighborhood $U$ of $u$
			and an element $a_u \in \qmod{G}$ such that
			\[
			\{x\in \mR^n: a_u(x) \geq 0\} \cap V_{\mR}(f-f_{\min}) \cap U \subset S(G).
			\]

			\bigskip
			\noindent
			{\bf Step 4.} During the above steps, we can get a  polynomial tuple
			$\Phi= \{\phi_1,\dots,\phi_s\}$ (see \reff{fin:phi}) with each
			$\phi_i\in \qmod{G}$ such that all conditions described
			at the beginning of this subsection are satisfied.
			Clearly, the quadratic module  $\qmod{\Phi}$ is finitely generated.
			Since $\qmod{\Phi}\subseteq \qmod{G}$, Theorem \ref{scal:bhc}
			can be applied to conclude $f-f_{\min}\in \qmod{G}$.
			This is shown in Section~\ref{proofthm1}.

			To prove Theorem~\ref{mianthm2}, we alternatively consider
			the scalar polynomial optimization problem
			\be  \label{n:pop1}
			\left\{ \baray{rl}
			\min & f(x)  \\
			\st &  	\phi(x) \geq  0 \, \, ( \, \phi \in \Phi). \\
			\earay \right.
			\ee
			Since each $\phi_i\in \qmod{G}$ and $f-f_{\min}\in \qmod{\Phi}$,
			the Moment-SOS hierarchy for solving \reff{n:pop1} has finite convergence
			and every minimizer of \reff{mom:local} also gives a minimizer of the moment relaxation
			for \reff{n:pop1}, by using subvectors. Then,
			we show that every minimizer of the moment relaxation for \reff{n:pop1}
			satisfies the flat truncation when the relaxation order is big enough.
			The proof of Theorem \ref{mianthm2} uses $f-f_{\min}\in \qmod{\Phi}$,
			which is given in the proof of Theorem~\ref{mianthm1}. Furthermore, we
			exploit some properties of moment matrices to complete the proof,
			which is of independent interest.

			\subsection{A locally equivalent reformulation}
			\label{sec:eqpmo}

			Let $u$ be a local minimizer of \reff{nsdp} with $\rank\, G(u) =r$.
			In this subsection, we give a locally equivalent reformulation for \reff{nsdp},
			in a neighborhood $U$ of $u$. This yields a reduced
			polynomial matrix  optimization problem,
			with the new psd constraint $T(x) \succeq 0$,
			for some $T \in \mathcal{S}^{m-r}\mathrm{QM}[G]$.
			By doing this, the ``$H$-term" in the second order sufficient condition \reff{sosc}
			disappears, while algebraic structures of $G(x)$ do not change.
			Furthermore, it also gives a convenient local parametrization,
			which is very useful for getting local SOS representations in Section~\ref{sec:localsos}.
			The new locally equivalent reformulation is motivated by the approach in \cite{fa}.

			Up to a permutation of rows and columns, one can write that
			$$
			G(x)=\left[\begin{array}{cc}
				A(x) & B(x) \\
				B(x)^T & C(x)
			\end{array}\right],
			$$
			where $A(x)\in \mathcal{S}^r\mR[x]$  and $\rank \, A(u) =r$.
			Thus, in a neighborhood $U$ of $u$, the matrix $A(x)$ is nonsingular.
			Denote the Schur complement:
			$$
			S(x):=C(x)-B(x)^T A(x)^{-1} B(x) .
			$$
			Then it holds that
			\be \label{defS}
			G(x)=\left[\begin{array}{cc}
				I_{r} & 0 \\
				B(x)^{T} A(x)^{-1} & I_{m-r}
			\end{array}\right]\left[\begin{array}{cc}
				A(x) & 0 \\
				0 & S(x)
			\end{array}\right]\left[\begin{array}{cc}
				I_{r} & A(x)^{-1} B(x) \\
				0 & I_{m-r}
			\end{array}\right],
			\ee
			for all $x$ in a neighborhood $U$ of $u$.
			Hence, we know that $G(x) \succeq 0$ is equivalent to $S(x)\succeq 0$ in $U$,
			and \reff{nsdp} is locally equivalent to
			\be  \label{re:nsdp}
			\left\{ \baray{rl}
			\min & f(x)  \\
			\st &  	S(x)\succeq 0,~x\in U. \\
			\earay \right.
			\ee
			The gradient operator $\nabla S(u)$ is said to be {\it regular}
			\cite{dgs}
			if  the adjoint mapping
			$\nabla S(u)^*: \mathcal{S}^{m-r} \rightarrow \mathbb{R}^n$
			is one-to-one, i.e.,
			$$
			\operatorname{ker} \nabla S(u)^*=\{0\} \subset \mathcal{S}^{m-r}.
			$$
			In particular, the regularity of $\nabla S(u)$ implies that
			\[
			\sigma(m-r) %%=(m-r)(m-r+1)/2
			= \binom{m-r}{2} \le  n.
			\]
			%The following result was proved in \cite{fa},
			%also referring to \cite{dgs} for new proofs.
			The following result appeared in \cite{dgs,fa}.

			\begin{thm}\label{reopt}
				Suppose $u$ is a local minimizer of \reff{nsdp}, \rev{and
					the nondegeneracy condition \reff{CQ}, strict complementarity condition \reff{SCC},
					second order sufficient condition \reff{sosc}  hold at $u$.} Then, $u$ is also a local minimizer of
				\reff{re:nsdp} and there exists a matrix $\Lambda \in \mathcal{S}^{m-r}$ such that
				\be \label{rkkt}
				\begin{array}{rcl}
					\nabla L_1(u) &=& 0, \\
					\Lambda \succeq 0, \quad  S(u) &\succeq& 0, \\
					\langle \Lambda,  S(u)\rangle &=& 0,
				\end{array}
				\ee
				where
				$
				L_1(x) \, \coloneqq \,   f(x)-\langle \Lambda,  S(x)\rangle .
				$
				Furthermore, it also holds:
				\bit
				
				\item [(i)] The  linear mapping $\nabla S(u)$ is regular.

				\item [(ii)] $\Lambda \succ 0$.

				\item [(iii)] For all $0\neq h \in \operatorname{ker} \nabla S(u)$,
				$h^{T} \nabla^2  L_1(u) h>0$.
				
				\eit
			\end{thm}

			\subsubsection{Equivalent polynomial matrix optimization}
			
			In optimization \reff{re:nsdp}, the matrix $S(x)$ is generally a rational function,
			so we may not have $S(x) \in \mathcal{S}^{m-r}\mathrm{QM}[G]$.  Multiplying $S(x)$ by
			$(\det\,A(x))^2$, we can get  a new  locally  equivalent reformulation of \reff{nsdp}. We show that the NDC, SCC and SOSC
			still hold for the new formulation.
			
			For convenience, denote
			\[
			p(x) = \det\, A(x),\quad T(x)=p(x)^2 S(x),
			\]
			\[
			Q(x)=p(x) \left[\begin{array}{cc}
				I_{r} & -A(x)^{-1} B(x) \\
				0 & I_{m-r}
			\end{array}\right].
			\]
			Note that $p(x)A(x)^{-1} =A(x)^{*}$, which is the adjoint matrix of $A(x)$.
			Since $A(x)^*$ is a polynomial matrix, we have $Q(x)\in \mR[x]^{m\times m}$. Using \reff{defS},
			one can verify  that
			\be \label{defT}
			Q(x)^T G(x)Q(x)=\left[\begin{array}{cc}
				p(x)^2	A(x) & 0 \\
				0 & T(x)
			\end{array}\right].
			\ee
			So, we have $T(x) \in \mathcal{S}^{m-r}\mathrm{QM}[G]$.
			Since $p(u)\neq 0$, we know that
			$G(x) \succeq 0$ is equivalent to $T(x) \succeq 0$ in a neighborhood $U$ of $u$.
			Thus, \reff{nsdp} is equivalent to
			\be  \label{re:nsdpt}
			\left\{ \baray{rl}
			\min & f(x)  \\
			\st &  	T(x)\succeq 0,~x\in U. \\
			\earay \right.
			\ee
			Note that
			$u$ is still a local minimizer of \reff{re:nsdpt}.

			The following  proposition  is the analogue of Theorem \ref{reopt}  for $T$, instead of $S$.
			
			\begin{prop} \label{propt}
				Suppose $u$ is a minimizer of \reff{nsdp}, \rev{and
					the nondegeneracy condition \reff{CQ}, strict complementarity condition \reff{SCC},
					second order sufficient condition \reff{sosc}  hold at $u$.}
				Then, there exits $\Theta \in \mathcal{S}^{m-r}$ such that
				\be \label{prkkt}
				\begin{array}{rcl}
					\nabla L^{re}(u) &=& 0, \\
					\Theta \succeq 0,\quad  T(u) &\succeq& 0, \\
					\langle \Theta,  T(u)\rangle &=& 0,
				\end{array}
				\ee
				where
				$
				L^{re}(x):=f(x)-\langle \Theta,  T(x)\rangle .
				$
				Furthermore, the following holds:
				\bit
				
				\item [(i)] The  linear mapping $\nabla T(u)$ is regular.
				
				\item [(ii)] $\Theta \succ 0$.
				
				\item [(iii)] For all $0\neq h \in \operatorname{ker} \nabla T(u)$,
				$h^{T} \nabla^2 L^{re}(u) h>0$.
				
				\eit
			\end{prop}

			%%%%%%%%%%%%%%%%%%%%%%%%%%%%%%%%%%
			\begin{proof}
				Let $\Lambda$ be as in \reff{rkkt} and denote 	
				$\Theta=\frac{1}{p(u)^2}\Lambda$. Note that
				\[
				\baray{cc}
				\nabla_{x_i} T(x)&=\nabla_{x_i} (p(x)^2S(x)) =
				2p(x) \nabla_{x_i}p(x)S(x)+p(x)^2\nabla_{x_i} S(x),
				\earay
				\]
				for  $i=1,\dots,n$. 	
				Since $S(u)=0$, we have
				\be \label{eqgra}
				\nabla_{x_i} T(u)=p(u)^2\nabla_{x_i} S(u),~i=1,\dots,n.
				\ee
				Hence, we get
				\[
				\baray{ll}
				\nabla L^{re}(u)=	\nabla f(u)-\nabla T(u)^*[\Theta]&=	\nabla f(u)-p(u)^2\nabla S(u)^*[\frac{1}{p(u)^2}\Lambda]\\
				&=\nabla f(u)-\nabla S(u)^*[\Lambda]\\
				&=0.\\
				\earay
				\]
				Clearly, $\Theta \succeq 0$ and
				\[
				\langle \Theta,  T(u)\rangle=\langle \frac{1}{p(u)^2}\Lambda,  p(u)^2 S(u)\rangle=0.
				\]
				Hence, the above $\Theta$ satisfies \reff{prkkt}.	
				
				\medskip
				\noindent
				(i)	Suppose $X\in \mathcal{S}^{m-r} $ satisfies $\nabla T(u)^*[X]=0$, i.e.,
				\be \label{3.2.1}
				\left\langle\nabla_{x_i} T(u), X\right\rangle=0,~  i=1,\dots,n.
				\ee
				The relation \reff{eqgra} implies that
				\[
				\left\langle\nabla_{x_i} S(u), X \right\rangle=0,~  i=1,\dots,n.
				\]
				The regularity of $\nabla_{x_i} S(u)$ implies $X=0$
				(see item (i) of Theorem \ref{reopt}). Hence,
				$\operatorname{ker} \nabla T(u)^*= \{0\}$
				and the  linear mapping $\nabla T(u)$ is regular.
				
				\medskip
				\noindent
				(ii) Since $\Lambda \succ 0$, we have
				\[
				\Theta=\frac{1}{p(u)^2}\Lambda\succ 0.
				\]
				
				\medskip
				\noindent
				(iii) For $h\in \mR^n$, we have
				\be
				\baray{rcl}
				\nabla T(u)[h]&=&\sum\limits_{i=1}^n h_i \nabla_{x_i} T(u)
				=p(u)^2 \sum\limits_{i=1}^n h_i \nabla_{x_i} S(u) \\
				&=& p(u)^2 \nabla S(u)[h].\\
				\earay
				\ee
				Since $p(u)\neq 0$, we know $\operatorname{ker} \nabla T(u)
				=\operatorname{ker} \nabla S(u)$.
				For $1\leq i, j\leq m-r$,
				\[
				\nabla^2 T_{ij} = 2S_{ij}\nabla p(\nabla p)^{T}+2pS_{ij}\nabla^2p+2p\cdot(\nabla p(\nabla S_{ij})^{T}
				+ \nabla S_{ij}\nabla p^{T})+p^2\nabla^2 S_{ij}.\\
				\]
				Since $S(u)=0$, it holds that
				\[
				\nabla^2 T_{ij}(u)=2p(u)(\nabla p(u)(\nabla S_{ij}(u))^{T}
				+ \nabla S_{ij}(u)\nabla p(u)^{T})+p(u)^2\nabla^2 S_{ij}(u).
				\]
				For $0\neq h \in \operatorname{ker} \nabla T(u)$, we  have
				\[
				\baray{rcl}
				h^{T} \nabla^2 L^{re}(u) h
				&=& h^{T}\nabla^2 f(u)h-\sum\limits_{i,j= 1}^{m-r} \Theta_{ij}h^{T}\nabla^2 T_{ij}(u)h\\
				&=& h^{T}\nabla^2 f(u)h-4p(u)\sum\limits_{i,j= 1}^{m-r}
				\Theta_{ij}(h^{T}\nabla p(u))(h^{T}\nabla S_{ij}(u))\\
				& & \qquad - p(u)^2\sum\limits_{i,j= 1}^{m-r} \Theta_{ij}h^{T}\nabla^2 S_{ij}(u)h\\
				&=& h^{T}\nabla^2 f(u)h- \sum\limits_{i,j= 1}^{m-r} \Lambda_{ij}h^{T}\nabla^2 S_{ij}(u)h\\
				&=& h^{T} \nabla^2 L_1(u) h>0,\\
				\earay
				\]
				where the third equality is due to $h \in \operatorname{ker} \nabla S(u)$.

			\end{proof}

			Since $T(u)=0$, the canonical unit vectors $e_1,\dots,e_{m-r} \in \mR^{m-r}$
			form a basis of $\operatorname{ker} T(u)$ (i.e., the null space of $T(u)$).
			The linear subspace defined in \reff{tangent} for \reff{re:nsdpt} at $u$ is
			\be
			\mc{N}(u)=\left\{h:=\left(h_1, \ldots, h_n\right) \in \mathbb{R}^n \mid
			\sum_{i=1}^n h_i \cdot \nabla_{x_i}  T(x) =0\right\}.
			\ee
			It implies that $\mc{N}(u) = \operatorname{ker} \nabla T(u)$.
			The regularity of the linear mapping $\nabla T(u)$ is equivalent to
			the nondegeneracy condition for \reff{re:nsdpt} at $u$.
			Thus, the items (i), (ii), (iii) are exactly the NDC, SCC, SOSC
			at the local minimizer $u$ of \reff{re:nsdpt}.

			\subsection{Local SOS representations}
			\label{sec:localsos}
			
			Recall that for a point $u \in \mR^n$ and  a quadratic module $M \subseteq \mR[x]$, the notation $\widehat{M}_u$ denotes the quadratic module
			generated by the image of $M$ in the completion ring of the localization of $\mR[x]$ at $u$ (see Section \ref{sec:qm} for the  notation).
			In this subsection, we show that if the NDC, SCC and SOSC
			hold at a minimizer $u$ of \reff{nsdp}, then $f-f_{\min}\in  \widehat{\qmod{G}}_u $. This gives a local SOS representation.

			Suppose $u$ is a local minimizer of \reff{nsdp} satisfying the NDC and $r = \rank\, G(u)$. In Section \ref{sec:eqpmo}, we know that
			\[
			\sigma(m-r) = \binom{m-r}{2} \, \le \, n.
			\]
			Let us introduce the vector $v=\left(v_1, \ldots, v_n\right)$ consisting of $n$ new variables.
			Let $V \in \mathcal{S}^{m-r} \mathbb{R}[v]$
			be the matrix whose upper triangular part is
			\be \label{def:V}
			\left[\begin{array}{clcl}
				v_1 &v_{2}&\dots &v_{ m-r } \\
				&v_{m-r+1}&\dots & v_{2(m-r)-1} \\
				& &\ddots & \vdots \\
				& & &v_{\sigma(m-r) } \\
			\end{array}\right].
			\ee
			The ring
			of formal power series in $v_1, \ldots, v_n$ is denoted by $\mathbb{R}\left[\left[v \right]\right]$ or $\mathbb{R}\left[\left[v_1, \ldots, v_n \right]\right]$.

			In the following, we first give the relationship between the Taylor expansion of $f$
			and its SOS representation in the completion ring $\widehat{\mR[v]}_u$.
			For convenience, assume $u=0$ and $f_{\min} =0$ at the moment.

			\begin{lem}\label{aule}
				Let  $0\leq r \leq n$ be an integer with $\sigma(m-r)\leq n$. Suppose that $f$ admits an expansion
				\be \label{expand0}
				f=\langle C,  V\rangle+f_2(v)+\sum\limits_{i=3}^{\infty}f_i  \in
				\mathbb{R}\left[\left[v_1, \ldots, v_n \right]\right],
				\ee
				where the symmetric matrix $C \in \mathcal{S}^{m-r}$ is positive definite, the quadratic form $f_2 \in \mathbb{R}[v]$  with
				$
				f_2(0,\ldots,0, v_{\sigma(m-r) +1},\ldots, v_n)
				$
				being positive definite in $v_{\sigma(m-r) +1},\dots, v_n$, and $f_i$ is the degree-$i$ homogeneous part of  $f$.
				Then, $f$ lies in the quadratic module in
				$\mathbb{R}[[v]]$, which is
				generated by the image of $\qmod{V}$, i.e., $f\in \widehat{\qmod{V}}_0$.	
			\end{lem}

			\begin{proof}
				We can rewrite $f$ as
				\[
				f= \langle C,  V\rangle+\langle H, V  \rangle +h =\langle C+H,  V\rangle +h,
				\]
				where $H=(h_{ij})$ is a symmetric matrix whose entries have no constant terms,
				and $h$ equals the sum of $ f_2\left(0, \ldots, 0, v_{\sigma(m-r) +1}, \ldots, v_n\right)$
				and terms of degree $3$ or higher in
				$v_{\sigma(m-r) +1}, \ldots, v_n$.
				Since $f_2(0,\ldots,0, v_{\sigma(m-r) +1},\ldots, v_n)$	is positive definite in  $v_{\sigma(m-r) +1},\dots, v_n$, there exist a symmetric matrix $A\succ 0$ and a  symmetric polynomial matrix $W(x)$   whose entries have no constant terms
				such that
				\[
				h = (v^{\prime})^T(A+W)v^{\prime},
				\]
				where $v^{\prime}=(v_{\sigma(m-r) +1}, \ldots, v_n)$.
				Consider the formal power series
				\[
				\omega(t) \, \coloneqq \,
				1+\frac{1}{2}t-\frac{1}{8}t^2+\cdots,
				\]
				which is the Taylor expansion of $\sqrt{1+t}$. Then, every entry of the matrix $A^{-\frac{1}{2}}WA^{-\frac{1}{2}}$ has no constant terms  and  we have that
				\[
				\baray{ll}
				h=(v^{\prime})^T(A+W)v^{\prime}&=(v^{\prime})^TA^{\frac{1}{2}}(I+A^{-\frac{1}{2}}WA^{-\frac{1}{2}})A^{\frac{1}{2}}v^{\prime}\\
				&=(\omega(A^{-\frac{1}{2}}WA^{-\frac{1}{2}})A^{\frac{1}{2}}v^{\prime})^T\omega(A^{-\frac{1}{2}}WA^{-\frac{1}{2}})A^{\frac{1}{2}}v^{\prime}.
				\earay
				\]
				It implies  that
				$h$ is SOS in $\mathbb{R}[[v]]$.
				This kind of representations are
				similarly used in \cite{mar06,Sch03}.

				%First, we show that $h$ is an SOS in $\re[[x]]$. Note that $f_2(0,\ldots,0, x_{\sigma(m-r) +1},\ldots, x_n)$ is positive definite.  Up to an invertible linear transformation, we assume that
				%\[
				%f_2\left(0, \ldots, 0, x_{\sigma(m-r) +1}, \ldots, x_n\right)= x_{\sigma(m-r) +1}^2+\cdots+x_n^2.
				%\]
				%Then, $h$ can be rewritten as
				%\[
				%\baray{ll}
				%h&=x_{\sigma(m-r) +1}^2+\cdots+x_n^2+x_{\sigma(m-r) +1}t_{\sigma(m-r) +1}+\cdots+x_{n}t_{n}\\
				%&=x_{\sigma(m-r) +1}^2(1+t_{\sigma(m-r)+1})+\cdots+x_{n}^2(1+t_{n}),\\
				%\earay
				%\]
				%where $t_{\sigma(m-r)+1},\dots, t_{n}$ are polynomials in $x_{\sigma(m-r) +1}, \ldots, x_n$.
				% Then, we have that
				%\[
				%h=x_{\sigma(m-r) +1}^2\cdot \ell(t_{\sigma(m-r)+1})^2+\cdots+x_{n}^2\cdot \ell(t_{n})^2,
				%\]
				%which implies that  $h$ is an SOS in $\re[[x]]$.
				%Clearly, $h$ is SOS in $\re[[x]]$
				%(cf. \cite[Lemma 2.1]{mar06}).
				%\textcolor{blue}{((The proof is not short, is it necessary?)}

				Case {\bf  I:} If $m-r=1$,  $C+H$ is a scalar polynomial. Then,
				we have 
				\[
				C+H=C(1+H/C)=C\omega(H/C)^2.
				\]
				Thus, $C+H$ is a sum of squares in $\re[[v]]$ and  $f\in \widehat{\qmod{V}}_0$.

				Case {\bf  II:}
				Consider the case $m-r>1$.

				Without loss of generality, we can assume  $C=I_{m-r}$.
				Otherwise, we can replace $V$ by $C^{\frac{1}{2}}VC^{\frac{1}{2}}$ and then
				the  multiplier matrix $C$ becomes $I_{m-r}$.
				One can write $ I_{m-r}+H$ as
				\[
				I_{m-r}+H  \, =  \, \sum\limits_{1\leq i<j\leq m-r} H^{(ij)},
				\]
				where each matrix $H^{(ij)}$ is given as
				\[
				H^{(ij)}_{ks}=\left\{  \baray{ll}
				\frac{1+h_{ks}}{m-r-1},&\text{if~} (k,s)=(i,i),~(j,j),\\
				h_{ks},&\text{if~} (k,s)=(i,j),~(j,i),\\
				0,&\text{otherwise}.\\
				\earay  \right.
				\]
				Denote the  polynomial matrix $V^{(ij)} \in  \mathbb{R}[x]^{2 \times 2}$:
				\[
				V^{(ij)}=\left[  \baray{ll}
				V_{ii}&V_{ij}\\
				V_{ij}&V_{jj}\\
				\earay  \right]
				\]
				Clearly, $V^{(ij)} \in \qmod{V}^{2}$, so
				\[
				\langle H^{(ij)},  V\rangle \, =  \, \frac{1+2h_{ii}}{2(m-r-1)} \cdot V_{ii}+(\frac{1+h_{jj}}{m-r-1}-2(m-r-1)h_{ij}^2)V_{jj}+p^TV^{(ij)}p,
				\]
				where
				$
				p = \sqrt{2(m-r-1)} \bbm  \frac{1}{2(m-r-1)} \,\,\,\,   h_{ij}  \ebm^T.
				$
				Similarly, one can show that
				\[
				\frac{1+2h_{ii}}{2(m-r-1)}, \quad
				\frac{1+h_{jj}}{m-r-1}-2(m-r-1)h_{ij}^2
				\]
				are SOS in $\re[[x]]$.
				Thus, each $\langle H^{(ij)},  V\rangle$ lies in the quadratic module in
				$\re[[x]]$ generated by $\qmod{V}$.
				This implies $f\in \widehat{\qmod{V}}_0$.  Since $T(x) \in \mathcal{S}^{m-r}\qmod{G}$,  we know
				$f$ lies in the quadratic module in $\re[[x]]$
				generated by $\qmod{G}$. Since the completion of $\re[x]_0$
				is exactly the formal power series ring $\re[[x]]$,
				we finally get $f \in \widehat{\qmod{G}}_0$.

			\end{proof}

			Suppose $u$ is a local minimizer of \reff{nsdp} and $\rank \, G(u) = r$.
			Let $T(x)$ be the  polynomial matrix as in  \reff{re:nsdpt}.
			Recall that the gradient operator $\nabla T(u)$
			is said to be regular if $\operatorname{ker} \nabla T(u)^*= \{0\} \subset \mathbb{S}^{m-r}$,
			i.e., $\nabla T(u)^*$ is injective.

			\begin{lem}\label{linear}
				Suppose $u$ is a local minimizer of \reff{nsdp} and $\rank\,G(u)=r$.
				Let $T(x)$ be the  polynomial matrix as in  \reff{re:nsdpt}.
				If the linear mapping $\nabla T(u)$ is regular, then the gradient vectors
				\[
				\nabla T_{ij}(u) \quad  (1\leq i\leq j\leq m-r)
				\]
				are linearly independent.
			\end{lem}
			\begin{proof}
				For the canonical unit vectors $e_1,\dots,e_{m-r}$ in $\mR^{m-r}$, the matrices
				\[
				\frac{1}{2}(e_ie_j^{T}+e_je_i^{T})  \quad  (1\leq i\leq j\leq m-r)
				\]
				are linearly independent in $\mathbb{S}^{m-r}$.	 		
				Since $\nabla T(u)$ is regular, the vectors
				\[
				\nabla T(u)^*[\frac{1}{2}(e_ie_j^{T}+e_je_i^{T})] \, = \,
				\nabla T_{ij}(u) \quad (1\leq i\leq j\leq m-r)
				\]
				are also linearly independent.

			\end{proof}

			In the following, we prove that if the NDC, SCC and SOSC hold at
			a local minimizer $u$ of \reff{nsdp}, then $f-f_{\min}  \in \widehat{\qmod{G}}_u $.
			
			\begin{thm} \label{mat:bhc}
				Suppose $u$ is a local minimizer of \reff{nsdp},
				\rev{and
					the nondegeneracy condition \reff{CQ}, strict complementarity condition \reff{SCC},
					second order sufficient condition \reff{sosc}  hold at $u$.} 
				Then, $f-f_{\min}$ belongs to the quadratic module in $\widehat{\mathbb{R}[x]}_{u}$
				generated by the image of $\qmod{G}$, i.e.,
				\[
				f-f_{\min}\in \widehat{\qmod{G}}_u .
				\]
				That is, there exist finitely many polynomials
				$g_u^{(1)},g_u^{(2)},\dots,g_u^{(t_u)} \in \qmod{G}$
				and $\sigma_u^{(1)},\dots,\sigma_u^{(t_u)}$
				which are SOS in  $\widehat{\mathbb{R}[x]}_{u}$, such that
				\[
				f-f_{\min}=\sigma_u^{(1)}g_u^{(1)}+\cdots+\sigma_u^{(t_u)}g_u^{(t_u)}.
				\]
			\end{thm}
			
			\begin{proof}
				By Proposition~\ref{propt},  there exist $T(x) \in \mathcal{S}^{m-r}\mathrm{QM}[G]$ and a neighborhood $U$
				of $u$
				such that $T(u)=0$,   \reff{nsdp} is equivalent to
				\be  \nonumber
				\left\{ \baray{rl}
				\min & f(x)  \\
				\st &  	T(x)\succeq 0, ~x\in U,\\
				\earay \right.
				\ee
				and there exists $ \Theta \in \mathcal{S}_{+}^{m-r}$ such that
				\[
				\nabla 	L^{re}(u)=0 ,·~\langle \Theta,  T(u)\rangle=0,
				\]
				where
				\[
				L^{re}(x)  \coloneqq  f(x)-\langle \Theta,  T(x)\rangle .
				\]
				Furthermore, the linear mapping $\nabla T(u)$ is regular,
				the Lagrange multiplier matrix $\Theta \succ 0$, and the following holds
				\be \label{sosclocal}
				h^{T} \nabla^2 L^{re}(u) h>0 \,\quad \text{for all}
				\quad  0 \ne h \in \operatorname{ker} \nabla T(u).
				\ee
				%Note that  the matrices  $\frac{e_ie_j^{T}+e_je_i^{T}}{2} \in \mathbb{S}^{m-r}$ ($1\leq i\leq j\leq n) $ is linearly independent.		
				%The regularity of $\nabla T(u)$ implies that
				By Lemma \ref{linear}, the gradient vectors
				\[
				\nabla T_{ij}(u) \quad (1\leq i\leq j\leq m-r)
				\]
				are  linearly independent. Up to an invertible linear coordinate transformation
				and shifting by a constant, we can generally assume that $f_{\min}=0$, $u=0$ and
				\[
				\bbm \nabla T_{11}(0) & \nabla T_{12}(0)
				&\cdots  &  \nabla T_{m-r,m-r}(0) \ebm  =
				\left[\begin{array}{l}
					I_{\sigma(m-r)} \\
					0  \\
				\end{array}\right].
				\]
				Hence, we get
				\[
				\operatorname{ker} \nabla T(0)=\left\{h=\left(h_1, \ldots, h_n\right) \in \mathbb{R}^n \mid h_1=\cdots=h_{\sigma(m-r)}=0\right\}.
				\]
				The SOSC at $u$ is equivalent to that the sub-Hessian
				\[
				\left(\frac{\partial^2 L^{re}(0)}{\partial x_i \partial x_j}
				\right)_{\sigma(m-r)+1 \leq i, j \leq n}
				\]
				is positive definite.
				Define the polynomial function
				\[
				\varphi(x)  \coloneqq  \left(\varphi_I(x), \varphi_{I I}(x)
				\right): \mathbb{R}^n \rightarrow \mathbb{R}^n,
				\]
				where
				\[
				\varphi_I(x)=\left[\begin{array}{l}
					T_{11}(x) \\
					T_{12}(x)\\
					\vdots \\
					T_{m-r,m-r}(x)
				\end{array}\right], \quad \varphi_{I I}(x)=\left[\begin{array}{l}
					x_{\sigma(m-r)+1} \\
					x_{\sigma(m-r)+2} \\
					\vdots \\
					x_{n}
				\end{array}\right].
				\]
				Clearly, $\varphi(0)=0$, and the Jacobian of $\varphi$ at 0 is the identity matrix $I_n$.
				By the Implicit Function Theorem, there exists a neighborhood $U^{\prime}\subseteq U$
				of $0$ such that the equation $v=\varphi(x)$
				defines a smooth function $x(v)=\varphi^{-1}(v)$.
				Write that
				\[
				\varphi^{-1}(v)  = (\varphi_1(v), \ldots, \varphi_n(v)).
				\]
				%Thus, $u=\left(u_1, \ldots, u_n\right)$ can serve as a coordinate system for $\mathbb{R}^n$ around 0 and $u=\varphi(x)$.
				In the $v$-coordinate system, $S(G) \cap U^{\prime}$
				can be equivalently described as $V\succeq 0$,
				where $V$ is defined as in \reff{def:V}.
				Define new functions
				\[
				F(v)  \, \coloneqq \, f\left(\varphi^{-1}(v)\right), \quad
				\widehat{L}^{re}(v) \, \coloneqq \,  L^{re}\left(\varphi^{-1}(v)\right)
				= F(v)-\langle \Theta,  V\rangle.
				\]
				Note that $\nabla_x L^{re}(0)=0$ implies $\nabla_v \widehat{L}^{re}(0)=0$,
				so it holds that
				\be \label{rev1}
				\begin{aligned}
					& \frac{\partial F(0)}{\partial V_{ij}}=2\Theta_{ij},  ~1\leq i\leq j \leq m-r,\\
					& \frac{\partial F(0)}{\partial v_k}=0,~k=\sigma(m-r)+1, \ldots, n.
				\end{aligned}
				\ee
				Expand $F(v)$ locally around 0 as
				$$
				F(v)=f_0(v)+f_1(v)+f_2(v)+f_3(v)+\cdots,
				$$
				where each $f_i$ is the degree-$i$  homogeneous part of  $f$  in $v$.
				The equation \reff{rev1} implies that
				$$
				f_0(v)=f(u)\geq 0,~f_1(v)=\langle \Theta,  V\rangle.
				$$
				For $v \in U^{\prime}$, it holds that
				\begin{eqnarray*}
					F(0, \ldots, 0, v_{\sigma(m-r)+1},\dots,v_n)
					& = & \widehat{L}^{re}(0, \ldots, 0, v_{\sigma(m-r)+1}, \ldots, v_{n}) \\
					& = & L^{re}(\varphi^{-1}(0, \ldots, 0, v_{\sigma(m-r)+1}, \ldots, v_{n})) .
				\end{eqnarray*}
				For all $1\leq i,j\leq n$, we have
				\begin{eqnarray*}
					\frac{\partial^2 \widehat{L}^{re}(v)}{\partial v_i \partial v_j}
					&=&\sum_{1 \leq k, s \leq n} \frac{\partial^2 L^{re}(x(v))}{\partial x_k \partial x_s} \frac{\partial \varphi^{-1}_k(v)}{\partial v_i} \frac{\partial \varphi^{-1}_s(v)}{\partial v_j}   \\
					& &  \qquad \quad  +\sum_{1 \leq k \leq n} \frac{\partial L^{re}(x(v))}{\partial x_k} \frac{\partial^2 \varphi^{-1}_k(v)}{\partial v_i \partial v_j} .
				\end{eqnarray*}
				Evaluating the above at $v=0$, we get $x(v)=0$ and
				$$
				\frac{\partial^2 \widehat{L}^{re}(0)}{\partial v_i \partial v_j}=\sum_{1 \leq k, s \leq n} \frac{\partial^2 L^{re}(0)}{\partial x_k \partial x_s} \frac{\partial \varphi^{-1}_k(0)}{\partial v_i} \frac{\partial \varphi^{-1}_s(0)}{\partial v_j}.
				$$
				Since
				\[
				\operatorname{Jac}(\varphi)\Big\vert_{x=0} =
				\operatorname{Jac}\left(\varphi^{-1}\right)\Big\vert_{v=0} = I_n.
				\]
				Therefore, for all $\sigma(m-r)+1 \leq i, j \leq n$, we have
				\[
				\left.\frac{\partial^2 f_2}{\partial v_i \partial v_j}\right|_{v=0}=
				\left.\frac{\partial^2 \widehat{L}^{re}}{\partial v_i \partial v_j}\right|_{v=0}=\left.\frac{\partial^2 L^{re}}{\partial x_i \partial x_j}\right|_{x=0} .
				\]
				Note that the sub-Hessian
				$$
				\left(\frac{\partial^2 L^{re}(0)}{\partial x_i \partial x_j}\right)_{\sigma(m-r)+1 \leq i, j \leq n}
				$$
				is positive definite. Thus, the quadratic form $f_2$ is positive definite in
				$x_{\sigma(m-r)+1}$, $\ldots$, $x_{n}$. By Lemma \ref{aule}, we know that
				$F(v)$ belongs to the quadratic module in $\re[[x]]$ generated by
				$\qmod{V}$. Thus,  $f-f_{\min}$ lies in the quadratic module in $\re[[x]]$
				generated by $\qmod{T}$. Since $T(x) \in \mathcal{S}^{m-r}\mathrm{QM}[G]$,  we know
				$f-f_{\min}$ lies in the quadratic module in $\re[[x]]$
				generated by $\qmod{G}$. Since the completion of $\re[x]_0$
				is exactly the formal power series ring $\re[[x]]$,
				we finally get $f-f_{\min} \in \widehat{\qmod{G}}_0$.

			\end{proof}

			\subsection{ Proof of  Theorem \ref{mianthm1} (also Theorem~\ref{mianthm1:local})}
			\label{proofthm1}
			First, we prove that if the optimality conditions
			hold at a minimizer $u$ of \reff{nsdp}, then   there are a neighborhood $U$
			of $u$  and an element $a_u \in \qmod{G}$
			such that
			\[
			\{x\in \mR^n: a_u(x) \geq 0\} \cap V_{\mR}(f-f_{\min}) \cap U \subset S(G).
			\]
			
			\begin{thm}\label{exist:a}
				Suppose $u$ is a local minimizer of \reff{nsdp},
				\rev{and
					the nondegeneracy condition \reff{CQ}, strict complementarity condition \reff{SCC},
					second order sufficient condition \reff{sosc}  hold at $u$.}  Then, there are a neighborhood $U$
				of $u$  and an element $a_u \in \qmod{G}$
				such that
				\[
				\{x: a_u(x) \geq 0\} \cap V_{\mR}(f-f_{\min}) \cap U \subset S(G).
				\]
			\end{thm}

			\begin{proof}	
				Up to  shifting by a constant,
				we can assume that  $f_{\min}=0$.
				Suppose $u$ is a minimizer of \reff{nsdp} and $\rank \, G(u) = r$.
				By Proposition~\ref{propt}, we know that there exist $T(x) \in \mathcal{S}^{m-r}\mathrm{QM}[G]$ and a neighborhood $U^{\prime}$
				of $u$
				such that $T(u)=0$,  $u$ is a local minimizer of
				\be  \label{locthm1}
				\left\{ \baray{rl}
				\min & f(x)  \\
				\st &  	T(x)\succeq 0,~x\in U^{\prime},  \\
				\earay \right.
				\ee
				and there exits $ \Theta \in \mathcal{S}_{+}^{m-r}$ such that
				\be  \nn %\label{muil:thm}
				\nabla 	L^{re}(u)=0 , \quad \langle \Theta,  T(u)\rangle=0,
				\ee
				where
				$$
				L^{re}(x) \, \coloneqq \,  f(x)-\langle \Theta,  T(x)\rangle .
				$$
				Furthermore, the  linear mapping $\nabla T(u)$ is regular.
				Let
				$X $ be the $(m-r)$-by-$(m-r)$ symmetric polynomial matrix whose upper triangular part is
				\be \nonumber
				\left[\begin{array}{clcl}
					x_1 &x_{2}&\dots &x_{ m-r } \\
					&x_{m-r+1}&\dots & x_{2(m-r)-1} \\
					& &\ddots & \vdots \\
					& & &x_{\sigma(m-r) } \\
				\end{array}\right],
				\ee
				and denote $y=(x_{\sigma(m-r)+1},\dots,x_n)$.
				Up to an invertible linear coordinate transformation,
				we can generally assume that  $u=0$, and
				%{\small
					%\[
					%\left[\begin{array}{cccc}
						%T_{11} &T_{m-r ~1}&\dots &T_{1 ~m-r} \\
						%&T_{22}&\dots & T_{2~m-r} \\
						%& &\ddots &\\
						%& & & T_{m-r ~m-r}\\
						%\end{array}\right]=\left[\begin{array}{cccc}
						%x_1+t_{11} &x_{m-r+1}+t_{12}&\dots &x_{\sigma(m-r)}+t_{1 ~m-r} \\
						%&x_{2}+t_{22}&\dots & x_{\sigma(m-r)-1}+ t_{2 ~m-r}\\
						%& &\ddots &\\
						%& & & x_{m-r}+t_{m-r ~m-r}\\
						%\end{array}\right],
						%\]
						%}
					\[
					T=X+T^{\prime} .
					\]
					Here, $T^{\prime}$ is a symmetric polynomial matrix and
					every entry of $T^{\prime}$ is a polynomial with all terms
					having degrees at least $2$. Furthermore, one can also assume that
					$\Theta=I_{m-r}$, because otherwise we can replace
					$T$ by $\Theta^{\frac{1}{2}}T\Theta^{\frac{1}{2}}$.
					By Theorem \ref{mat:bhc}, we have that
					%\textcolor{red}{(change $r$ to $s$?)}
					\[
					f=\langle I_{m-r},  T\rangle+f_2(y)+\langle H,T\rangle + s,
					\]
					where the quadratic form $f_2(y)$
					%$$
					%h:=f^{(2)}(0, \ldots, 0, x_{\sigma(m-r)+1}, \ldots, x_n)-\sum_{i=1}^{m-r}  T_{ii}^{(2)}(0, \ldots, 0, x_{\sigma(m-r)+1}, \ldots, x_n)
					%$$
					is positive definite in $y$,
					%By assumption we have
					%$$
					%f(x,y) = \sum_{i=1}^m \langle I,X\rangle + q(y) + \langle H,X\rangle+r_1 ,
					%$$
					%where $r_1$ is the summation of all terms of order at least three, with $C \succ  0$ and
					$H$ is a symmetric polynomial matrix with each entry being
					linear forms, and $s$ denotes the sum of all terms with degrees at least $3$.
					%where
					%\[
					%q(y) = \sum_{j,k=1}^n b_{jk}y_jy_k,
					%\]
					%is a positive definite quadratic form.
					Denote
					\be \label{def:au}
					a_u \, = \, \operatorname{Trace}\Big( (I_{m-r}+(H-T)/2)^2T \Big).
					\ee
					Then, we have that
					\be
					\baray{rcl}
					f-a_u &=&\langle I_{m-r},  T\rangle+f_2(y)+\langle H,T\rangle+s
					-\operatorname{Trace}\Big((I_{m-r}+(H-T)/2)^2T \Big)\\
					&=& f_2(y)+\langle T,T\rangle+s-\frac{1}{4}\operatorname{Trace}((H-T)^2T).\\
					\earay
					\ee
					Equivalently, we have that
					\[
					f \, = \, a_u + f_2(y)+\langle X,X\rangle + s^{\prime},
					\]
					where all terms of $s^{\prime}$ have degrees $3$ or higher.
					Since $T(x) \in \mathcal{S}^{m-r}\qmod{G}$, we know that $a_u \in \qmod{G}$.

					Since the form $f_2(y)+\langle X,X\rangle$ is positive definite,
					there exists a neighborhood $U\subseteq U^{\prime}$ of $0$
					such that $f_2(y)+\langle X,X\rangle+s^{\prime}>0$
					for all $0 \ne x \in U$.
					In the following, we show that the polynomial $a_u$ given in the above satisfies
					\[
					\{x\in \mR^n: a_u(x) \geq 0\} \cap V_{\mR}(f-f_{\min}) \cap U = \{ 0 \} \subseteq S(G).
					\]
					%where  $U$ is a neighborhood of $u$.
					Suppose otherwise there exists $0 \neq q\in U$ such that
					\[
					a_u(q)\geq 0,~f(q)=f_{\min}=0.
					\]
					Then, we can get the contradiction
					\[
					0=f(q)>a_u(q)\geq 0.
					\]		
					
				\end{proof}

				In the following, we give the complete proof of  Theorem \ref{mianthm1}
				(also Theorem~\ref{mianthm1:local}), with the usage of results shown in Sections \ref{sec:eqpmo}, \ref{sec:localsos}.
				
				\bigskip
				\noindent
				{\bf Proof of  Theorem \ref{mianthm1} (also Theorem~\ref{mianthm1:local}).}
				Suppose $u$ is a local minimizer of \reff{nsdp}.	By Theorem \ref{mat:bhc}, we know
				$f - f_{\min}  \in \widehat{\qmod{G}}_u$, i.e.,  there exists a finite set of polynomials
				\[
				g_u=\{g_u^{(1)},g_u^{(2)},\dots,g_u^{(t_u)}\} \subseteq \qmod{G}
				\]
				such that $f$ belongs to the quadratic module in $\widehat{\mathbb{R}[x]}_{u}$
				generated by the image of $g_u$. Here, $t_u$ is a positive integer.
				Since the NDC, SCC and SOSC are assumed to hold at every minimizer of \reff{nsdp},
				these minimizers must be isolated. Hence, $f$ has only finitely many zeros in $S(G)$, say, $u^{(1)},u^{(2)},\dots,u^{(N)}$. Denote the polynomial tuple
				\[
				\phi = \bigcup_{i=1}^N  g_{u^{(i)}}\cup \{ a_{u^{(i)}} \}.
				\]
				Here, the polynomials $a_{u^{(i)}}$ are the analogues of $a_u$ as in \reff{def:au}.
				By %Proposition \ref{finitebas},
				\cite[Proposition 9]{sk09},
				there exist finitely many
				$\psi_1,\psi_2,\dots,\psi_{\ell} \in \qmod{G}$ such that
				\[
				S(G) = \{x\in \mR^n:\,  \psi_1(x) \geq 0, \psi_2(x) \geq 0,\dots, \psi_{\ell}(x) \geq 0\}.
				\]
				Since $\qmod{G}$ is \rev{Archimedean},
				there exists $R>0$ such that $R-\|x\|^2\in \qmod{G}$. Denote
				\be \label{fin:phi}
				\Phi \coloneqq \phi \cup \{\psi_1,\dots,\psi_{\ell}\}\cup \{R-\|x\|^2\} .
				\ee
				Clearly, the quadratic module $\qmod{\Phi}$ is \rev{Archimedean} and finitely generated.
				Since $\Phi \subseteq \qmod{G}$, we have
				\[
				\qmod{\Phi}\subseteq \qmod{G}, \quad
				S(G)=\{x\in \mR^n : \phi(x)\geq 0, \,\, \forall \, \phi \in \Phi \}.
				\]
				By  Theorems \ref{mat:bhc} and
				\ref{exist:a}, one can see that the assumptions in Theorem \ref{scal:bhc} are satisfied
				for the quadratic module $\qmod{\Phi}$ and the set $S(G)$. Hence,
				%the conditions in Proposition \ref{scal:bhc} are all satisfied for the quadratic module $ M$,
				it follows from Theorem \ref{scal:bhc} that $f - f_{\min} \in \qmod{\Phi}$. The inclusion $ \qmod{\Phi}\subseteq \qmod{G}$
				implies  $f - f_{\min} \in \qmod{G}$,
				which completes the proof.

				%
				%In the following, we give the proof of Theorem \ref{mianthm1} (Theorem \ref{mianthm1:local}).
				%
				%\subsection{ Proof of  Theorem \ref{mianthm1} (also Theorem~\ref{mianthm1:local})}

				%
				%In the following, we give the proof of Theorem \ref{mianthm2}
				%(Theorem \ref{mianthm2:local}).
				%

				%\bigskip
				%\bigskip
				\subsection{Proof of  Theorem~\ref{mianthm2} (also Theorem~\ref{mianthm2:local})}\label{proofthm2}
				Without loss of generality, we can assume that  $f_{\min}=0$, for convenience of notation.
				Let $\Phi= \{\phi_1,\dots,\phi_s\}$ be the polynomial set as in \reff{fin:phi}
				in the proof of Theorem \ref{mianthm1}.

				(i) Note that $\qmod{\Phi}\subseteq \qmod{G}$ and $\supp{\qmod{\Phi}+\ideal{f-f_{\min}}}$ %$(M+\ideal{f}) \cap-(M+\ideal{f})$
				is zero dimensional, by Theorem~\ref{scal:bhc}. Since
				\[
				\supp{\qmod{\Phi}+\ideal{f-f_{\min}}}\subseteq  \supp{\qmod{G}+\ideal{f-f_{\min}}},
				\]
				we know $\supp{\qmod{G}+\ideal{f}}$ is also zero dimensional.

				(ii)
				Consider the polynomial optimization problem
				\be  \label{n:pop}
				\left\{ \baray{rl}
				\min & f(x)  \\
				\st &  	\phi(x) \geq  0 \, \, ( \forall \, \phi \in \Phi). \\
				\earay \right.
				\ee
				Note that
				$
				S(G)=\{x\in \mR^n : \phi(x) \geq 0 \,  (\forall \, \phi \in \Phi) \}.
				$
				Hence, the optimal value of \reff{n:pop} is also $f_{\min}$.
				The $k$th order SOS relaxation for solving \reff{n:pop} is
				\be \label{sos:opt:stan}
				\left\{\begin{array}{cl}
					\max & \gamma \\
					\text { s.t. } & f-\gamma \in \mathrm{QM}[\Phi]_{ 2k} .
				\end{array}\right.
				\ee
				The dual optimization of \reff{sos:opt:stan} is the $k$th order moment relaxation:
				\be  \label{mom:opt:stan}
				\left\{ \baray{cl}
				\min &  \langle f, w \rangle  \\
				\st   &L_{\phi}^{(k)}[w] \succeq 0 \,\,  (\forall \, \phi \in \Phi) \\
				& M_k[w] \succeq 0, \\
				&  w_0  =  1,  \,  w \in \mathbb{R}^{\mathbb{N}_{2 k}^{n}} .
				\earay \right.
				\ee
				Let $\vartheta_{k}$, $\vartheta_{k}^{\prime}$
				denote the optimal value of \reff{sos:opt:stan}, \reff{mom:opt:stan}, respectively.
				By Theorem \ref{mianthm1}, we have $f-f_{\min}\in \qmod{\Phi}$, so
				$\vartheta_{k}=\vartheta_{k}^{\prime}=f_{\min}=0 $ for all $k\geq k_0$,
				for some order $k_0$.  Since every minimizer of \reff{mom} also gives a minimizer of the moment relaxation \reff{mom:opt:stan}, by using subvectors, it is enough to show that every minimizer of  \reff{mom:opt:stan} has a flat truncation when $k$ is big enough.
				
				Let $w^{(k)}$ be a minimizer of \reff{mom:opt:stan} at the relaxation order $k$.
				%Note that $\Phi \subseteq  \qmod{G}$. Let $k_0$ be sufficiently large such that
				%$\Phi \subseteq \qmod{G}_{k_0}$.
				%Suppose $w^{(k)}$ is a minimizer of the $k$th order moment relaxation \reff{mom},
				%then $\langle f ,  w^{(k)}\rangle = 0$.  Hence, the truncation $w^{(k)}\mid_{2k-2k_0}$
				%is also feasible for the $(k-k_0)$th order relaxation \reff{mom:opt:stan},
				%so $w^{(k)}\mid_{2k-2k_0}$ is also a minimizer of \reff{mom:opt:stan},
				%when $k$ is sufficiently enough.
				%\textcolor{red}{(Why do we need to do this?
					%The $< \cdot, \cdot >$ avoids these issues automatically?)}
				%For notational convenience,
				%we identify $w^{(k)}\mid_{2k-2k_0}$ as $w^{(k)}$.
				Denote
				\[
				Q \, \coloneqq \,  \qmod{\Phi}+\ideal{f-f_{\min}}.
				\]
				Note that $Q$ is also \rev{Archimedean} and the intersection $J:=Q \cap-Q$
				is an ideal. By item (i),
				%the coordinate ring $\frac{\mathbb{R}[x]}{J}$ has dimension 0.
				the ideal $J$ is zero dimensional.
				%\textcolor{red}{(In the item (i), don't we have already shown $\dim J  = 0$?
					%Why do we cite Proposition~\ref{scal:bhc} again?)}
				Denote the zero set
				$$
				Z  \coloneqq  \left\{\begin{array}{ll}
					x \in \mathbb{R}^n \left| \begin{array}{l}
						f(x)=0, \\
						\phi(x) \geq 0 \,\, (\forall \, \phi \in \Phi)
					\end{array}\right.
				\end{array}\right\} .
				$$
				Clearly, $Z$ is the set of all minimizers of  \reff{nsdp}.
				The SOSC implies that the minimizers of \reff{nsdp} are isolated.
				Since $S(G)$ is compact, the set $Z$ is finite and
				the vanishing ideal $I(Z)$ is zero dimensional.
				Let $\left\{h_1, \ldots, h_r\right\}$ be a Gr\"{o}bner basis of $I(Z)$
				with respect to a total degree ordering.
				Since $h_t \equiv 0$  on $Z$ $(t=1,\dots,r)$ and $J$ is zero dimensional,
				by \cite[Corollary~7.4.2]{marshall2008positive},
				there exist $\ell\in \mathbb{N}$, $\eta \in \mR[x]$ and $\psi_0, \psi_i \in \Sigma[x]$ such that
				$$
				h_t^{2 \ell}+\eta  f+
				\sum\limits_{i=1}^s \psi_i \phi_i+\psi_0=0 .
				$$
				When $2 k$ is bigger than the degrees of all above polynomials, we have
				\be \label{realzero}
				\langle h_t^{2 \ell}, w^{(k)}\rangle+\langle\eta  f, w^{(k)}\rangle+\sum\limits_{i=1}^s\langle \psi_i \phi_i, w^{(k)}\rangle+\langle\psi_0, w^{(k)}\rangle=0.
				\ee
				Since $f\in \qmod{\Phi}$,   there exits   $\psi_0^{*}$, $\psi_i^{*}  \in \Sigma[x]_{2k_0}$ such that
				\[
				\deg(\psi_i^{*}\phi_i)\leq 2k_0 ~~(i=1,\dots,s ),
				\]
				\be \label{per}
				f=\sum\limits_{i=1}^s \psi_i^{*} \phi_i+\psi_0^{*} .
				\ee
				Applying the bi-linear operation $\langle \cdot, w^{(k)}\rangle$
				to \reff{per} for $k> k_0$, we get
				\be
				0=\sum\limits_{i=1}^s \langle \psi_i^{*} \phi_i, w^{(k)}\rangle+\langle \psi_0^{*}, w^{(k)}\rangle.
				\ee
				Since  $\psi_0^{*}$,
				$\psi_j^{*} $ are SOS, the constraints
				\[
				L_{\phi_{i}}^{(k)}[w] \succeq 0~(i=1,\dots,s), \quad M_k[w] \succeq 0
				\]
				imply that
				\[
				\langle \psi_i^{*} \phi_i, w^{(k)}\rangle\geq0 ~(i=1,\dots,s) , ~
				\langle \psi_0^{*}, w^{(k)}\rangle \geq0.
				\]
				Hence, we  have that
				\[
				\langle \psi_i^{*} \phi_i, w^{(k)}\rangle =0 ~(i=1,\dots,s) , ~
				\langle \psi_0^{*}, w^{(k)}\rangle =0.
				\]
				%By Lemma 2.5 of \cite{nie2013certifying}, we can get for $k$ big enough,
				%\be \label{asym1}
				%\lim\limits_{\epsilon\rightarrow 0}\langle \psi_j^{\epsilon} c_j\eta , w^{(k)}\rangle=0, ~\lim\limits_{\epsilon\rightarrow 0}\langle \psi_0^{\epsilon}\eta , w^{(k)}\rangle=0.
				%\ee
				By Lemma 2.5 of \cite{nie2013certifying}, we can get for $k$ big enough,
				\be
				\langle \psi_i^{*}
				\phi_i\eta , w^{(k)}\rangle=0~(i=1,\dots,s),\,\, \langle \psi_0^{*}\eta, w^{(k)}\rangle=0.
				\ee
				Hence, we can get
				$$
				\langle \eta  f, w^{(k)}\rangle
				= \sum\limits_{i=1}^s \langle \psi_i^{*}
				\phi_i\eta , w^{(k)}\rangle+\langle \psi_0^{*}\eta , w^{(k)}\rangle
				=  0.
				$$
				It follows from \reff{realzero} that
				\be\label{realzero1}
				\langle h_t^{2 \ell}, w^{(k)}\rangle+\sum\limits_{i=1}^s\langle \psi_i \phi_i, w^{(k)}\rangle+\langle\psi_0, w^{(k)}\rangle=0.
				\ee
				Since $h_t^{2 \ell}$, $\psi_i$ ($i=1,\dots,s $), $\psi_0$ are SOS, we can easily verify that
				\[
				\langle h_t^{2 \ell}, w^{(k)}\rangle\geq 0,
				\langle\psi_0, w^{(k)}\rangle \geq0,~\langle \psi_i \phi_i, w^{(k)}\rangle \geq 0 ~ (i=1,\dots,s).
				\]
				Hence, $\langle h_t^{2 \ell}, w^{(k)}\rangle=0$ and
				$h_t \in \operatorname{ker} M_k[w^{(k)}]$
				(see \cite{LLR08,Lau09} or \cite[Lemma~4.2.6]{niebook}).
				
				\bigskip
				Next, we show that $\left.w^{(k)}\right|_{2 k-2}$ is flat.
				Note that $Z$ is finite. By Theorem 2.6 of \cite{Lau09},  the quotient space $\mathbb{R}[x] / I(Z)$
				is finite-dimensional, and
				let $\left\{q_1, \ldots, q_{\ell} \right\}$ be a  basis of $\mathbb{R}[x] / I(Z)$.
				For arbitrary $\alpha \in \mathbb{N}^n$, there exist $\beta_i\in \mR$, $p_t\in \mR[x]  $ such that
				$$
				x^\alpha=\sum_{i=1}^{\ell} \beta_i q_i+\sum_{t=1}^r p_t h_t,
				\quad \operatorname{deg}\left(p_t h_t\right) \leq|\alpha|.
				$$
				When $ |\alpha| \leq k-1$, we have
				\[
				p_t h_t \in \operatorname{ker}  M_k[w^{(k)}], \quad
				x^\alpha-\sum_{i=1}^{\ell} \beta_i q_i \in \operatorname{ker} M_k[w^{(k)}],
				\]
				since $h_t \in \operatorname{ker}  M_k[w^{(k)}]$.
				Let
				\[
				d_q:=\max\limits_{i \in [\ell]} \operatorname{deg} (q_i).
				\]
				Note that each $\alpha$th column  of $ M_k[w^{(k)}]$
				is a linear combination of its first $d_q$ columns for $d_q+1 \leq|\alpha| \leq k-1$. It implies that when $k-1-d_G \geq d_q$, we have
				\[
				\operatorname{rank} M_{k-1-d_G}[w^{(k)}]=\operatorname{rank} M_{k-1}[w^{(k)}].
				\]
				Thus,  $w^{(k)}$ has a flat truncation when $k$ is big enough.

				\section{Conclusions}
				\label{sec:dis}
				
				Under the \rev{Archimedean property}, we show that the  matrix Moment-SOS hierarchy
				of relaxations \reff{sos}--\reff{mom} has finite convergence,
				if the nondegeneracy condition, strict complementarity condition,
				and second order sufficient condition hold at every minimizer of \reff{nsdp}.
				Furthermore, under the same assumptions, we prove that every minimizer of
				the moment relaxation \reff{nsdp} has a flat truncation,
				when the relaxation order is sufficiently large.
				These results give connections between nonlinear semidefinite optimization theory
				and polynomial matrix optimization.

				\bigskip
				\noindent
				{\bf Acknowledgements.}
				The authors would like to thank the editors and anonymous referees
				for their fruitful comments and suggestions.
				They also thank Claus Scheiderer very much
				for personal communications about nonnegative polynomials.

				%%%%%%%%%%%%%%%%%%%%%
				\iffalse
				\ter{
					In practice,  the Moment-SOS relaxations often have finite convergence at the very low degree relaxations.
					For polynomial optimization, the geometric exactness at a given order is recently studied by Didier Henrion \cite{hd23}. While there exist nice results  for quadratic programming, how to provide suitably sufficient conditions for the finite convergence of low degree relaxations is a very important future work.
				}
				\textcolor{blue}{(It is not necessary to write the above.
					We can drop them if there are no good OPEN questions.)}
				
				\fi
				%%%%%%%%%%%%%%%%%%%%%%%%
				%\ter{
					%It is worth pointing out that all results in this
					%paper can be  extended to deal with separate equality
					%constraints by adapting the techniques recently developed in  \cite{hny2}, but we omit them for simplicity since the proof would be tedious.
					%}

\end{document}